\documentclass[11pt]{article}

\usepackage{ifthen}
\newboolean{notikz}
\setboolean{notikz}{false}
\newboolean{arxiv}
\setboolean{arxiv}{true}
\newboolean{optimalitycondition}
\setboolean{optimalitycondition}{true}
\ifthenelse{\boolean{arxiv}}
{
\pdfoutput=1  
\usepackage{amsthm}

\newcommand{\PARstart}[2]{#1#2}
}
{
}

\usepackage[utf8]{inputenc}
\usepackage{graphicx}
\usepackage{url}

\usepackage{amsmath} 

\usepackage{amssymb}  

\usepackage{hyperref}
\usepackage{algorithm}
\usepackage{algorithmic}

\newtheorem{thm}{Theorem}

\newtheorem{prop}{Proposition}

\ifthenelse{\boolean{arxiv}}
{\theoremstyle{definition}}
{}
\newtheorem{defn}{Definition}

\newtheorem{remark}{Remark}
\newtheorem{example}{Example}

\newcommand{\NEW}[1]{{\em #1}}
\newcommand{\R}{\mathbb{R}}
\def\proofarg[#1]{\noindent\hspace{2em}{\itshape Proof of #1: }}

\newcommand{\sR}{\mathcal{R}}

\newcommand{\supp}{\mathrm{supp}}
\newcommand{\extr}{\mathrm{extr}}
\newcommand{\co}{\mathrm{co}}
\newcommand{\abs}[1]{\lvert #1 \rvert}
\newcommand{\norm}[1]{\lVert #1 \rVert}
\newcommand{\quotepage}[1]{{\tt\small #1}}

\ifthenelse{\boolean{notikz}}
{}
{
\usepackage{tikz}
\usetikzlibrary{snakes,arrows,shapes}
}

\title{
Ergodic Control and Polyhedral approaches to PageRank Optimization
}

\ifthenelse{\boolean{arxiv}}
{
\author{Olivier Fercoq \footnotemark[1] \footnotemark[3]\and
Marianne Akian \footnotemark[1] \and
Mustapha Bouhtou \footnotemark[2] \and
St\'ephane Gaubert \footnotemark[1] 
}
}
{
\author{Olivier Fercoq*, Marianne Akian*, Mustapha Bouhtou** and St\'ephane Gaubert*%
\thanks{*INRIA Saclay and CMAP Ecole Polytechnique}
\thanks{{\tt\small olivier.fercoq@inria.fr, 
marianne.akian@inria.fr,
stephane.gaubert@inria.fr 
}}%
\thanks{**Orange Labs}
\thanks{{\tt\small mustapha.bouhtou@orange-ftgroup.com}}%
\thanks{The doctoral work of the first author is supported by Orange Labs through the research contract CRE~3795 with INRIA}%
}
}

\begin{document}

\maketitle
\ifthenelse{\boolean{arxiv}}
{
 \footnotetext[1]{INRIA Saclay and CMAP Ecole Polytechnique\\
 {\tt\small olivier.fercoq@inria.fr}\\
 {\tt\small marianne.akian@inria.fr}\\
 {\tt\small stephane.gaubert@inria.fr}}%
 \footnotetext[2]{France Télécom R \& D \\
 {\tt\small mustapha.bouhtou@orange-ftgroup.com}}%
 \footnotetext[3]{The doctoral work of the first author is supported by Orange Labs through the research contract CRE~3795 with INRIA}%
}
{}

\begin{abstract}
We study a~general class of PageRank optimization problems
which consist in finding an~optimal outlink 
strategy for a~web site subject to design constraints.
We consider both a continuous problem,
in which one can choose the intensity of a~link,
and a~discrete one, in which in each page, there are obligatory
links, facultative links and forbidden links.
We show that the continuous problem, as well as its discrete variant
when there are no constraints coupling different pages, can
both be modeled by constrained Markov 
decision processes with ergodic reward, in which the webmaster 
determines the transition probabilities of websurfers.
Although the number of actions turns out to be exponential, we show
that an associated polytope of transition measures
has a concise representation, from which we deduce that 
the continuous problem is solvable in polynomial time, 
and that the same is true for the discrete problem
when there are no coupling constraints.
We also provide efficient algorithms, adapted to very large networks.
Then, we investigate the qualitative features of optimal outlink strategies,
and identify in particular assumptions under which there exists a~``master'' page to which all controlled pages should 
point. We report numerical results on fragments of the real web graph.
 \end{abstract}

\section{Introduction}

\PARstart{T}{he PageRank} introduced by Brin and Page~\cite{Brin-Anatomy}
is defined as the invariant measure of a~walk made by a~random surfer
on the web graph. When reading a~given page, the surfer either
selects a~link from the current page (with a~uniform probability),
and moves to the page pointed by that link, or interrupts his current search,
and then moves to an~arbitrary page, which is selected
according to given ``zapping'' probabilities. 
The rank of a~page is defined as its frequency of visit
by the random surfer.

The interest of the PageRank algorithm is to give each page of the web 
a measure of its popularity. It is a~link-based measure, 
meaning that it only takes into account the hyperlinks 
between web pages, and not their content. 
It is combined in practice with content-dependent measures, 
taking into 
account the relevance of the text of the page to the query of the user, 
in order 
to determine the order in which the answer pages will be shown
by the search engine. This leads to a~family of search methods the details of which 
may vary (and are often not publicly known).  
However, a~general feature of these methods is that among the pages  
with a~comparable relevance to a~query, the ones with the highest PageRank will appear first.  

The importance of optimizing the PageRank, 
specially for e-business purposes, has led to the development of a~number of companies offering 
Search Engine Optimization services. 
We refer in particular the reader to~\cite{seo-bestpractice} 
 for a~discussion of the PageRank optimization methods which are used in practice.
Understanding PageRank optimization is also useful to fight
malicious behaviors like link spamming, which intend to increase
artificially the PageRank of a~web page~\cite{linkspamalliances}, \cite{collusionTopologies}. 


The 
PageRank has motivated a number of works, dealing
in particular with computational issues.
Classically, the PageRank vector is computed by the power algorithm~\cite{Brin-Anatomy}. 
There has been a~considerable work on designing new, more efficient approaches for 
its computation~\cite{PRcomputing, LanMey-Beyond}: Gauss-Seidel method~\cite{Arasu-GaussSeidelPageRank}, 
aggregation/disaggregation~\cite{LanMey-Beyond} 
or distributed randomized algorithms~\cite{IshiiTempo-DistrRandAlg, Nazin-DistrRandAlg}. Other active fields are 
the development of new ranking algorithms~\cite{Borodin-linkanalysis} or the study of the web graph~\cite{Bonato-webgraph}. 



The optimization of PageRank has 
been studied
by several authors.
Avrachenkov and Litvak analyzed in~\cite{AvrLit-OptStrat}
the case of a~single controlled page and 
determined an~optimal strategy.
In~\cite{Viennot-2006}, Mathieu and Viennot established several bounds
indicating to what extent the rank of the pages of a~(multi-page) website can be changed,
and derived an~optimal referencing strategy in a~special unconstrained case:
if the webmaster can fix arbitrarily the hyperlinks in a~web site,
then, it is optimal to delete every link pointing outside
the web site.
To avoid such degenerate strategies, De Kerchove, Ninove and van Dooren~\cite{NinKer-PRopt} studied
the~problem of maximizing the sum of the PageRank coordinates in a~web site, provided that from each page, there is at least
one path consisting of hyperlinks and leading to an~external page. 
They gave a~necessary structural condition satisfied by an~optimal outlink
strategy.
In~\cite{Nin-PhD}, Ninove developed a~heuristic based on these theoretical
results, which was experimentally shown to be efficient. 
In~\cite{IshiiTempo-FragileDataPRcomp}, Ishii and Tempo investigated
the sensitivity of the PageRank to fragile (i.e.~erroneous or imperfectly known) web data, including fragile links (servers not responding, links
to deleted pages, etc.). They~gave bounds on the possible
variation of PageRank and introduced an~approximate PageRank
optimization problem, which they showed to be equivalent
to a~linear program. 
In~\cite{Blondel-PRopt}, (see also~\cite{Blondel-PRoptArxiv} for more details),
Csáji, Jungers and Blondel thought of fragile links as controlled links and
gave an~algorithm
to optimize in polynomial time the PageRank of a~single page.

In the present paper, we study 
a more general PageRank optimization problem, in
which a~webmaster, controlling a~set of pages (her web site), wishes
to maximize a~utility function depending on the PageRank or, more
generally, on the associated \NEW{occupation measure} (frequencies of visit
of every link, the latter are more informative).
For instance, the webmaster
might wish to maximize the number of clicks per time unit of a~certain
hyperlink bringing an~income, or the rank of the most visible page of her site,
or the sum of the ranks of the pages of this site, etc.
We consider specifically two versions of the PageRank optimization problem.

We first study a~{\em continuous} 
 version of the problem in which the set of actions of the webmaster is the set of admissible transition probabilities
of websurfers.
This means that the webmaster, by choosing the importance of the hyperlinks 
of the pages she controls (size of font, color, position of the link within the page),
determines a~continuum of possible transition probabilities.
Although this model has been already proposed
by Nemirovsky and Avrachenkov~\cite{Avra-WeightedPageRank},
its optimization does not seem to have considered previously.
 This continuous version includes rather realistic
constraints: for instance, the webmaster may start from
a ``template'' or ``skeleton'' (given by designers),
and be allowed to modify this skeleton only to a~limited extent.
Moreover,
we shall allow {\em coupling constraints} between
different pages (for instance, the rank of one page may be required
to be greater than the rank of another page, constraints involving the sum
of the pageranks of a~subset of pages are also allowed, etc.). 

Following~\cite{IshiiTempo-FragileDataPRcomp,Blondel-PRopt}, we also study a~{\em discrete} version of the problem, in which in each page, there are obligatory
links, facultative links and forbidden links. Then, the decision
consists in selecting the subset of facultative links which are
actually included in the page. 

We show that when there are no coupling constraints between different
pages and when the utility function is linear, the continuous 
 and discrete
problems both can be solved in polynomial time by reduction to a
linear program (our first main result, Theorem~\ref{thm:polytime}).
When specialized to the discrete problem,
this extends
Theorem~1 of~\cite{Blondel-PRopt}, which only applies
to the case 
in which 
the utility function represents the PageRank of a~single page. 
The proof of Theorem~\ref{thm:polytime} 
relies on the observation that
the polytope generated by the transition probability measures that
are uniform on some subsets of pages
has a~concise representation with a polynomial
number of facets (Theorem~\ref{thm:nonEmptySet}).
This leads us to prove a general result of independent interest
concerning Markov decision processes
with implicitly defined action sets. We introduce the notion of {\em well-described} Markov decision processes, in which, although
there may be an {\em exponential} number of actions, there is a polynomial time {\em strong separation oracle} for the actions polytope (whereas the classical complexity results assume
that the actions are explicitly enumerated~\cite{Papadimitriou-complMDP}).
We prove in Theorem~\ref{thm:polytime+}, as an application of 
the theory of Khachiyan's ellipsoid method (see~\cite{Lovasz-geomAlgo}), that the ergodic control problem for well-described Markov decision process is polynomial time solvable
(even in the multi-chain framework). Then, Theorem~\ref{thm:polytime}
follows as a direct corollary. We note that maximization or separation oracles have been previously considered in dynamic programming for different purposes (dealing with unnkown parameters~\cite{Givan-boundedMDP,huanxu-robustMDP}, or approximating large scale problems~\cite{kveton}). 

Proposition~\ref{prop:contrrate} yields a~fixed point scheme with a~contraction rate independent of the number of pages. Indeed, the contraction rate
depends only on the ``damping factor'' (probability
that the user interrupts his current search).
Therefore, this problem can be solved efficiently for very large instances by Markov decision techniques. Our results show that optimizing the PageRank is not much more difficult than computing it, provided there are no coupling constraints: indeed, Proposition~\ref{prop:greedy} shows that by comparison, the execution time is only increased by a~$\log n$ factor, where $n$ is the number of pages.
Note that the Markov decision process which we construct here is quite different
from the one of~\cite{Blondel-PRopt}, the latter
is a stochastic shortest path problem, whose construction
is based on a~graph rewriting technique,
in which intermediate (dummy) nodes are added to the graph. 
Such nodes are not subject to damping and therefore, 
the power iteration looses its 
uniform contraction.
In our approach, we use a more general ergodic control model, which allows us to consider a general linear utility function, and avoids
adding such extra nodes.
Experiments also show that the present approach leads to a~faster algorithm (Section~\ref{sec:expDisc}).

We also study the continuous 
problem with general (linear) coupling constraints,
and show that the latter can also be solved in polynomial time by reduction
to a~{\em constrained} ergodic control problem. Proposition~\ref{prop:Lag} 
yields an algorithm to solve the PageRank optimization problem
with coupling constraints, which scales well if the number of coupling constraints remains small.
The resolution uses Lagrangian relaxation and convex programming techniques like the bundle method.
There is 
little hope to solve efficiently, in general, the discrete problem with general coupling constraints since Csáji, Jungers and Blondel have proved in~\cite{Blondel-PRopt}
that the discrete PageRank optimization problem with mutual exclusion constraints is NP-complete.
Nevertheless, we develop a~heuristic for the discrete PageRank optimization
problem with linear coupling constraints, based on the optimal
solution of a relaxed continuous problem (Section~\ref{sec:expCoupl}).
On test instances, approximate optimality certificates show that the
solution found by the heuristic is at most at 1.7\% of the optimum.

Using the concept of mean reward before teleportation, 
we identify in Theorem~\ref{thm:masterPageDisc} (our second main result) assumptions 
under which there exists a~``master'' page to which all controlled pages should point.
The theorem gives an~ordering of the pages such that
in loose terms, the optimal strategy is at each page to point to the allowed pages with highest order.
The structure of the obtained optimal website is somehow reminiscent of Theorem~12 in~\cite{NinKer-PRopt}, 
but in~\cite{NinKer-PRopt}, there is only one constraint: the result is thus different.
\ifthenelse{\boolean{optimalitycondition}}
{
When the problem has coupling constraints, the mean reward before teleportation
still gives information on optimal strategies (Theorem~\ref{thm:condopt}).
}
{}

We report numerical results on the web site of one
of the authors (including an~aggregation of surrounding pages) as well as on a~fragment of the web ($4.10^5$ pages from the universities of New Zealand).

We finally note that an~early Markov Decision Model for PageRank optimization
was introduced by Bouhtou and Gaubert in 2007, in the course of the supervision of the student project of Vlasceanu and Winkler~\cite{Vla-Hyp}.

The paper is organized as follows. In Section~\ref{sec:mod}, 
we introduce the general PageRank optimization problem. 
In Section~\ref{sec:polytope}, we give a~concise description 
of the polytope of uniform transition probabilities.
In Section~\ref{sec:solvelocal}, we show that every Markov decision
process which admits such a concise description is polynomial time
solvable (Theorem~\ref{thm:polytime+}), and we deduce as a corollary our first main result, Theorem~\ref{thm:polytime}.
Section~\ref{subsec-value} describes an~efficient fixed point scheme 
for the resolution of the PageRank optimization problem with local constraints. 
In Section~\ref{sec:shape}, we give the ''master page'' Theorem (Theorem~\ref{thm:masterPageDisc}). 
We deal with coupling constraints in Section~\ref{sec:coupling}. 
We give experimental results on real data in Section~\ref{sec:exp}. 

\section{PageRank optimization problems} \label{sec:mod}

\subsection{Google's PageRank}

We first recall the basic elements of the Google PageRank computation, 
see~\cite{Brin-Anatomy} and~\cite{LanMey-Beyond} for more information.
We call \NEW{web graph} the directed graph with a~node per web page
and
an arc from page~$i$ to page~$j$ if page~$i$ contains a~hyperlink to page~$j$.
We identify the set of pages to $[n]:=\{1,\ldots,n\}$. 
\ifthenelse{\boolean{arxiv}}{

}
{
}
Let $N_i$ denote the number of hyperlinks contained in page $i$.
Assume first that $N_i\geq 1$ for all $i\in [n]$, meaning that every page
has at least one outlink. Then, we construct
the $n\times n$ stochastic matrix $S$, which is such that
\begin{align}
S_{i,j}=\begin{cases}N_i^{-1} & \text{if page }j \text{ is pointed to from page }i\\
0 & \text{otherwise}
\end{cases}\label{e-def-S}
\end{align}
This is the transition matrix of a~Markov chain modeling the behavior of
a surfer choosing a~link at random, uniformly among the ones included in the current page and moving to the page pointed by this link.
The matrix $S$ only depends of the web graph. 

We also fix a~row vector $z \in \R_+^n $, the \NEW{zapping}
or \NEW{teleportation} vector, which must be stochastic (so, $\sum_{j\in[n]} z_j =1$), 
together with a~\NEW{damping factor} $\alpha \in [0,1]$ and define the new stochastic matrix
\begin{equation*}
 P=\alpha S + (1-\alpha) e z
\end{equation*}
where $e$ is the (column) vector in $\R^n$ with all entries equal to~$1$. 

Consider now a~Markov chain
$(X_t)_{t\geq0}$ with transition matrix $P$, so that for all $i,j\in [n]$, 
$\mathbb{P}(X_{t+1}=j | X_t=i) = P_{i,j}$.
Then,
$X_t$ represents the position of a~websurfer at time $t$:
when at page $i$, the websurfer continues his current
exploration of the web with probability $\alpha$ and moves to the next
page by following the links included in page $i$, as above, or with
probability $1-\alpha$, stops his current exploration and then teleports
to page $j$ with probability~$z_j$.

When some page $i$ has no outlink, $N_i=0$, and so the entries of the $i$th row of the matrix $S$ cannot be defined according to~\eqref{e-def-S}.
Then, we set $S_{i,j}:=z_j$. In other words, when visiting a~page without
any outlink, the websurfer interrupts its current exploration and teleports
to page~$j$ again with probability $z_j$. It is also possible to define
another probability vector~$Z$ (different from $z$) for the teleportation from these ``dangling nodes''. 

The \NEW{PageRank} $\pi$ is defined as the invariant measure of 
the Markov chain $(X_t)_{t\geq 0}$ representing the behavior of the websurfer.
This invariant measure is unique if $\alpha<1$, or if $P$ is irreducible.

Typically, one takes $\alpha=0.85$, meaning that at each step, a~websurfer
interrupts his current search with probability $0.15\simeq 1/7$.
The advantages of the introduction of the damping factor
and of the teleportation vector are well known. First, it guarantees
that the power algorithm converges to the PageRank with a~geometric
rate $\alpha$ independent of the size (and other characteristics) of
the web graph. In addition, the teleportation vector may be used
to ``tune'' the PageRank if necessary. By default, $z=e^T/n$
is the uniform stochastic vector.
 We will assume in the sequel that $\alpha<1$ and $z_j>0$ for all $j\in [n]$,
so that $P$ is irreducible.

The graph on Figure~\ref{graph:base} represents a~fragment of the web graph.
We obtained the graph
by performing a~crawl of our laboratory with 1500 pages. We set the teleportation vector in such a~way that 
the 5 surrounding institutional pages are dominant.
The teleportation probabilities to these pages were taken to be proportional to the
PageRank (we used the Google Toolbar, which gives a~rough indication
of the PageRank, on a~logarithmic scale).
After running the PageRank algorithm on this graph, we found that within 
the controlled site, the main page of this author has the biggest PageRank
(consistently with the results provided by Google search).

\begin{figure}[thbp]
  \centering
\ifthenelse{\boolean{notikz}}
{
\includegraphics[width=23em]{graphExample}
}
{
\small{
\begin{tikzpicture}[>=latex', join=bevel, scale=0.32]
\pgfsetlinewidth{0.5bp}
\pgfsetcolor{black}
  \draw [->] (599bp,454bp) .. controls (621bp,437bp) and (664bp,406bp)  .. (704bp,376bp);
  \draw [->] (332bp,510bp) .. controls (342bp,510bp) and (350bp,508bp)  .. (350bp,503bp) .. controls (350bp,500bp) and (347bp,498bp)  .. (332bp,496bp);
  \draw [->] (463bp,516bp) .. controls (438bp,514bp) and (381bp,510bp)  .. (332bp,505bp);
  \draw [->] (332bp,499bp) .. controls (385bp,492bp) and (500bp,476bp)  .. (564bp,468bp);
  \draw [->] (129bp,305bp) .. controls (169bp,350bp) and (238bp,430bp)  .. (283bp,481bp);
  \draw [->] (467bp,507bp) .. controls (454bp,493bp) and (430bp,469bp)  .. (401bp,439bp);
  \draw [->] (569bp,478bp) .. controls (541bp,504bp) and (481bp,561bp)  .. (441bp,599bp);
  \draw [->] (497bp,318bp) .. controls (519bp,295bp) and (560bp,251bp)  .. (601bp,208bp);
  \draw [->] (329bp,490bp) .. controls (350bp,479bp) and (380bp,461bp)  .. (402bp,440bp) .. controls (427bp,417bp) and (449bp,386bp)  .. (469bp,355bp);
  \draw [->] (409bp,599bp) .. controls (389bp,581bp) and (356bp,552bp)  .. (325bp,523bp);
  \draw [->] (504bp,336bp) .. controls (544bp,339bp) and (629bp,343bp)  .. (696bp,347bp);
  \draw [->] (330bp,514bp) .. controls (371bp,531bp) and (447bp,564bp)  .. (513bp,591bp);
  \draw [->] (419bp,594bp) .. controls (410bp,564bp) and (393bp,509bp)  .. (376bp,454bp);
  \draw [->] (295bp,475bp) .. controls (280bp,415bp) and (242bp,276bp)  .. (215bp,176bp);
  \draw [->] (571bp,450bp) .. controls (554bp,428bp) and (523bp,389bp)  .. (496bp,354bp);
  \draw [->] (325bp,523bp) .. controls (346bp,543bp) and (380bp,572bp)  .. (409bp,600bp);
  \draw [->] (446bp,614bp) .. controls (460bp,614bp) and (479bp,615bp)  .. (508bp,615bp);
  \draw [->] (230bp,675bp) .. controls (248bp,632bp) and (272bp,576bp)  .. (291bp,530bp);
  \draw [->] (435bp,595bp) .. controls (443bp,579bp) and (456bp,555bp)  .. (470bp,529bp);
  \draw [->] (460bp,345bp) .. controls (449bp,350bp) and (437bp,356bp)  .. (414bp,367bp);
  \draw [->] (115bp,532bp) .. controls (161bp,525bp) and (224bp,516bp)  .. (274bp,507bp);
  \draw [->] (291bp,530bp) .. controls (277bp,563bp) and (253bp,620bp)  .. (230bp,675bp);
  \draw [->] (488bp,358bp) .. controls (501bp,398bp) and (529bp,485bp)  .. (552bp,556bp);
  \draw [->] (581bp,485bp) .. controls (580bp,500bp) and (579bp,522bp)  .. (576bp,553bp);
  \draw [->] (565bp,459bp) .. controls (535bp,450bp) and (475bp,431bp)  .. (417bp,413bp);
  \draw [->] (274bp,508bp) .. controls (238bp,514bp) and (173bp,524bp)  .. (114bp,532bp);
  \draw [->] (463bp,320bp) .. controls (423bp,288bp) and (325bp,210bp)  .. (252bp,151bp);
  \draw [->] (445bp,621bp) .. controls (455bp,621bp) and (464bp,619bp)  .. (464bp,614bp) .. controls (464bp,611bp) and (460bp,609bp)  .. (445bp,607bp);
  \draw [->] (329bp,451bp) .. controls (327bp,456bp) and (324bp,462bp)  .. (316bp,476bp);
  \draw [->] (564bp,468bp) .. controls (518bp,474bp) and (403bp,489bp)  .. (332bp,499bp);
  \draw [->] (458bp,331bp) .. controls (398bp,321bp) and (238bp,295bp)  .. (143bp,279bp);
  \draw [->] (480bp,311bp) .. controls (478bp,279bp) and (475bp,219bp)  .. (471bp,159bp);
  \draw [->] (566bp,473bp) .. controls (547bp,482bp) and (517bp,497bp)  .. (488bp,511bp);
  \draw [->] (503bp,342bp) .. controls (513bp,342bp) and (522bp,340bp)  .. (522bp,335bp) .. controls (522bp,332bp) and (518bp,330bp)  .. (503bp,328bp);
  \draw [->] (284bp,481bp) .. controls (250bp,442bp) and (179bp,362bp)  .. (129bp,305bp);
  \draw [->] (316bp,477bp) .. controls (319bp,472bp) and (321bp,466bp)  .. (329bp,451bp);
  \draw [->] (485bp,527bp) .. controls (494bp,536bp) and (506bp,549bp)  .. (527bp,570bp);
  \draw [->] (469bp,355bp) .. controls (455bp,378bp) and (430bp,414bp)  .. (402bp,440bp) .. controls (383bp,458bp) and (358bp,474bp)  .. (329bp,490bp);
\begin{scope}
  \pgfsetstrokecolor{black}
  \pgfsetfillcolor{lightgray}
  \filldraw (476bp,517bp) ellipse (13bp and 14bp);
  \draw (476bp,517bp) node {\href{http://amadeus.inria.fr/gaubert/PAPERS/TU-0190}{a}};
\end{scope}
\begin{scope}
  \pgfsetstrokecolor{black}
  \draw (466bp,80bp) ellipse (79bp and 79bp);
  \draw (466bp,80bp) node {\href{http://www.siam.org/meetings/ct09}{conference}};
\end{scope}
\begin{scope}
  \pgfsetstrokecolor{black}
  \pgfsetfillcolor{lightgray}
  \filldraw (481bp,335bp) ellipse (23bp and 24bp);
  \draw (481bp,335bp) node {\href{http://amadeus.inria.fr/gaubert/programme/index.html}{c}};
\end{scope}
\begin{scope}
  \pgfsetstrokecolor{black}
  \pgfsetfillcolor{lightgray}
  \filldraw (425bp,614bp) ellipse (21bp and 21bp);
  \draw (425bp,614bp) node {\href{http://amadeus.inria.fr/gaubert/HOWARD2.html}{b}};
\end{scope}
\begin{scope}
  \pgfsetstrokecolor{black}
  \pgfsetfillcolor{lightgray}
  \filldraw (583bp,465bp) ellipse (19bp and 20bp);
  \draw (583bp,465bp) node {\href{http://amadeus.inria.fr/gaubert/papers.html}{d}};
\end{scope}
\begin{scope}
  \pgfsetstrokecolor{black}
  \draw (357bp,394bp) ellipse (63bp and 63bp);
  \draw (357bp,394bp) node {\href{http://www.inria.fr}{institute}};
\end{scope}
\begin{scope}
  \pgfsetstrokecolor{black}
  \draw (650bp,156bp) ellipse (71bp and 71bp);
  \draw (650bp,156bp) node {\href{http://www.lsv.ens-cachan.fr/Seminaires/anciens?l=fr&sem=200811041100}{university}};
\end{scope}
\begin{scope}
  \pgfsetstrokecolor{black}
  \draw (742bp,349bp) ellipse (46bp and 46bp);
  \draw (742bp,349bp) node {\href{http://dx.doi.org/10.1145/1190095.1190110}{arxiv}};
\end{scope}
\begin{scope}
  \pgfsetstrokecolor{black}
  \pgfsetfillcolor{lightgray}
  \filldraw (303bp,503bp) ellipse (29bp and 30bp);
  \draw (303bp,503bp) node {\href{http://amadeus.inria.fr/gaubert}{main}};
\end{scope}
\begin{scope}
  \pgfsetstrokecolor{black}
  \draw (58bp,541bp) ellipse (57bp and 57bp);
  \draw (58bp,541bp) node {\href{http://www.cmap.polytechnique.fr}{lab}};
\end{scope}
\begin{scope}
  \pgfsetstrokecolor{black}
  \draw (571bp,616bp) ellipse (63bp and 63bp);
  \draw (571bp,616bp) node {\href{http://users.dsic.upv.es/~sas2008/accepted_papers.html}{other}};
\end{scope}
\begin{scope}
  \pgfsetstrokecolor{black}
  \draw (100bp,272bp) ellipse (44bp and 44bp);
  \draw (100bp,272bp) node {\href{http://www.cmap.polytechnique.fr/spip.php?rubrique103}{team}};
\end{scope}
\begin{scope}
  \pgfsetstrokecolor{black}
  \draw (196bp,107bp) ellipse (71bp and 71bp);
  \draw (196bp,107bp) node {\href{http://www-rocq.inria.fr/metalau/cohen}{friend2}};
\end{scope}
\begin{scope}
  \pgfsetstrokecolor{black}
  \draw (203bp,737bp) ellipse (67bp and 67bp);
  \draw (203bp,737bp) node {\href{http://www-rocq.inria.fr/metalau/quadrat}{friend1}};
\end{scope}
\end{tikzpicture}
}
}
  \caption{The web site of one of the authors (colored) and the surrounding sites (white). This 1500-page fragment of the web is aggregated for presentation,
using the technique described in~\cite{LanMey-Beyond}. 
The sizes of the circles follow the log of their PageRank. 
}
  \label{graph:base}
\end{figure}

\subsection{Optimization of PageRank}  \label{sec:var}

The problem we are interested in is the optimization of PageRank. 
We study two versions of this problem. 
In the continuous PageRank Optimization problem, the 
webmaster can choose the importance of the hyperlinks of the 
pages she controls and thus she has a continuum of admissible transition probabilities (determined for instance by selecting the color of a hyperlink,
the size of a font, or the position of a hyperlink in a page). 
This continuous model is specially 
useful in e-business applications, in which the income depends on the
effective frequency of visit of pages by the users, rather than on
its approximation provided by Google's pagerank.
The Continuous PageRank Optimization Problem is given by:
\begin{equation} \label{eqn:ideal}
  \max_{\pi,P} \left \{ U(\pi,P) \; \; ; \; \pi = \pi  P ,\quad \pi \in \Sigma_n, \quad P \in \mathcal{P} \right \}
\end{equation}
Here, $\Sigma_n:=\{x\in \R^n\mid x_i\geq 0,\forall i\in[n];\;\sum_{i\in[n]}x_i=1\}$ is
 the \NEW{simplex} of dimension $n$, $U$ is a~\NEW{utility} function and
$\mathcal{P}$ 
is a set representing the set of all admissible
 transition probability matrices. 
We denote by $P_{i,\cdot}$ the $i$th row of a~matrix $P$. We shall
distinguish \NEW{local constraints}, which can be expressed
as $P_{i,\cdot}\in \mathcal{P}_i$,
where $\mathcal{P}_i\subset \Sigma_n$ is a~given subset, and \NEW{global constraints},
which couple several vectors $P_{i,\cdot}$. Thus, local constraints only involve
the outlinks from a~single page, whereas global constraints involve
the outlinks from different pages.
We shall consider the situation in which each $\mathcal{P}_i$
is a polytope (or more generally 
an effective convex set).



If we restrict our attention to Google's PageRank (with uniform
transition probabilities), we arrive at the following
combinatorial optimization problem.
For each page $i$, as in~\cite{IshiiTempo-FragileDataPRcomp} and~\cite{Blondel-PRopt},
we partition the set of potential links $(i,j)$ into three subsets, consisting respectively of \NEW{obligatory links}~$\mathcal{O}_i$,
\NEW{prohibited links}~$\mathcal{I}_i$ and the set of \NEW{facultative links}~$\mathcal{F}_i$.
Then, for each page~$i$, we must select the subset $J_i$ of the set of
facultative links $\mathcal{F}_i$ which are effectively included
in this page. Once this choice is made for every page, we get
a new webgraph, and define the transition matrix $S=S(J_1,\ldots,J_n)$ as in%
~\eqref{e-def-S}. The matrix after teleportation is also
defined as above by
$P(J_1,\ldots,J_n):= \alpha S(J_1,\ldots,J_n)+(1-\alpha)ez$.
Then, the 
Discrete PageRank Optimization Problem is given by:
\begin{equation} \label{eqn:idealDisc} 
  \max_{\pi,P} \{ U(\pi, P) \; \; ; \; \pi = \pi  P ,\quad \pi \in \Sigma_n, \quad P=P(J_1,\ldots,J_n),\; J_i \subseteq \mathcal{F}_i, 
\; i\in [n]
 \} 
\end{equation} 

\begin{remark} \label{rem:combinatorialPb}
Problem~\eqref{eqn:idealDisc} is a combinatorial optimization problem: if there are $p_i$ facultative
links in page $i$, the decision variable, $(J_1,\ldots,J_n)$, takes 
$2^{p}$ values, where $p=p_1+\dots+p_n$.
\end{remark}

We shall be specially interested in the modeling of an income proportional to the frequency of clicks on some hyperlinks. 
Let $r_{i,j}$ be a~reward per click for each hyperlink $(i,j)$.
The latter utility can be represented by the following {\em linear utility function},
which gives the total income:
\begin{equation} \label{eqn:income}
 U(\pi,P)= \sum_{i\in[n]}\pi_i \sum_{j\in[n]} P_{i,j} r_{i,j}  \enspace. 
\end{equation}
Unless stated otherwise, we will consider the total income linear utility in the sequel.

\begin{remark}
The problem of maximizing the total PageRank of a~web site 
(sum of the PageRanks of its pages) is obtained as a~special case
of~\eqref{eqn:income}.
Indeed, if this web site consists of the subset of pages
$I \subseteq [n]$, one can set $r_{i,j}= \chi_I(i), \forall i,j \in [n]$,
where $\chi_I$ is the characteristic function of $I$ (with value $1$ if $i\in I$ and $0$ otherwise). Then
$U(\pi,P)= \sum_i \pi_i \sum_j P_{i,j} r_{i,j} = \sum_{i \in I} \pi_i$.
\end{remark}

\begin{remark} \label{rem:teleportPrice}
Note that the general form of the utility function assumes
that we receive the same instantaneous reward $r_{i,j}$ when the surfer follows the hyperlink $(i,j)$
 and when the surfer stops the current exploration at page $i$ to teleport to page $j$. There is no
 loss of generality in assuming that it is so: 
assume that the surfer produces a~reward
of $r'_{i,j}$ when he follows the hyperlink $(i,j)$ and $0$ when he teleports
to page $j$. 
Using the fact that $\sum_{j \in [n]} r'_{i,j} z_j = \sum_{j \in [n]} \sum_{l \in [n]} r'_{i,l} z_l P_{i,j}$ and $P=\alpha S + (1-\alpha) e z$,
we show that $\alpha \sum_{i,j \in [n]} r'_{i,j} \pi_i S_{i,j} = \sum_{i,j \in [n]} ( r'_{i,j} - (1-\alpha) \sum_{l \in [n]} r'_{i,l} z_l ) \pi_i P_{i,j}$.
 We then only need to set $r_{i,j}= r'_{i,j} - (1-\alpha) \sum_{l \in [n]} r'_{i,l} z_l$.
\end{remark}

We shall restrict our attention to situations in which $\pi$ is uniquely defined for each admissible transition matrix $P\in \mathcal{P}$ (recall that this is the case in particular when $\alpha<1$). Then the utility $U$ is a~function of $P$ only.

Alternatively, it will be convenient to think of the utility as a~function of
the \NEW{occupation measure} $\rho=(\rho_{i,j})_{i,j\in[n]}$.
The latter is the stationary distribution
of the Markov chain $(x_{t-1},x_{t})$. Thus, $\rho_{i,j}$ gives the
frequency of the move from page~$i$ to page~$j$.
The occupation measure $\rho$ is a~probability measure and it satisfies the flow relation, so that
\begin{equation}
\rho_{i,j} \geq 0,\; \forall i,j\in[n] \;,\qquad \sum_{i,j\in[n]} \rho_{i,j} =1\;,
\qquad
\sum_{k\in[n]} \rho_{k,i} = \sum_{j\in[n]} \rho_{i,j} ,\; \forall i\in [n]
\enspace .\label{e-flow}
\end{equation}
The occupation measure may also be thought of as a~matrix. Hence, we shall
say that $\rho$ is \NEW{irreducible} when the corresponding matrix is irreducible. 

The occupation measure $\rho$ can be obtained from the invariant measure
$\pi$ and the stochastic matrix $P$ by $\rho_{i,j} = \pi_i P_{i,j}, \forall i,j\in[n]$ and, conversely, the invariant measure $\pi$ can be recovered from $\rho$ by $\pi_i = \sum_{j\in [n]} \rho_{i,j}, \forall i\in[n]$.

The map $f$ which determines the stochastic matrix $P$ from
the occupation measure is given~by:
\begin{align}
P=f(\rho), \qquad
P_{i,j} &= \frac{\rho_{i,j}}{\sum_{k} \rho_{i,k}} ,\qquad \forall i,j\in[n]\label{e-def-f} \enspace .
\end{align}

\begin{prop} \label{prop:homeo}
The function $f$ defined by~\eqref{e-def-f} sets up a~birational transformation between the set of irreducible occupation measures 
(irreducible matrices satisfying~\eqref{e-flow}) 
 and the set of irreducible stochastic matrices.
In particular, the Jacobian of $f$ is invertible at any point of the set of irreducible occupation measures.
\end{prop}
\begin{proof}
As $\pi$ is uniquely defined,
its entries are a~rational function of the entries of $P$
(for instance, when $P$ is irreducible, an~explicit rational expression is given by Tutte's Matrix Tree Theorem~\cite{Tutte}).
The invertibility of the Jacobian follows from the birational character of $f$.
\end{proof}

This bijective correspondence will allow us to consider the occupation measure,
rather than the stochastic matrix $P$, as the decision variable. 
Note that the utility function can be written as a linear function in terms of the 
occupation measure:
$ 
U(\pi,P)=\sum_{i,j\in[n]}\rho_{i,j} r_{i,j}
$. 

\subsection{Design constraints of the webmaster} \label{sec:cont}

We now model the possible modifications made by the webmaster, who may
be subject to constraints imposed by the designer of the web site
(the optimization of the PageRank should respect the primary
goal of the web site, which is in general to offer some content).
We thus describe the set $\mathcal{P}$ of admissible transition probabilities
of~\eqref{eqn:ideal}.

\begin{prop} \label{prop:convex+}
 Assume that $\mathcal{P} = \prod_{i \in [n]} \mathcal{P}_i$, that for all $i \in [n]$, $\mathcal{P}_i$ is a~closed convex and that every matrix $P \in \mathcal{P}$
is irreducible.
Then, the set $\mathcal{R}$ of occupation measures arising from the elements of $\mathcal{P}$
is also a~closed convex set. Moreover, if every $\mathcal{P}_i$ is a~polytope, then so is $\mathcal{R}$.
\end{prop}
\begin{proof}
 For all $i \in [n]$, $\mathcal{P}_i$ is a~closed convex set and so it is the intersection 
of a~possibly infinite family of hyperplanes $(H_i^{(l)})_{l \in L}$.
Every element $P$ of $\prod_{i \in [n]} \mathcal{P}_i$
must satisfy the following inequalities, 
one for each $H_i^{(l)}$:
\begin{equation}
  \sum_{j \in [n]} a_{i,j}^{(l)} P_{i,j} \leq b_i^{(l)}, \qquad \forall i \in [n], \forall l \in L
\label{e-affl}
\end{equation}
Formulating these equalities in terms of the~occupation measure $\rho$ thanks to $P_{i,j} = \frac{\rho_{i,j}}{\sum_{j'} \rho_{i,j'}}$ 
and Proposition~\ref{prop:homeo}, and rewriting Inequalities~\eqref{e-affl}
in the form
\begin{equation}
\sum_{j\in [n]} a_{i,j}^{(l)} \rho_{i,j} \leq b_i^{(l)} \sum_{k\in[n]} \rho_{i,k}, \qquad \forall i \in [n], \forall l \in L
\label{e-ineq-liftrho}
\end{equation}
we see that $\rho$ satisfies
a family of constraints of the form~\eqref{e-ineq-liftrho}, 
together with the inequalities~\eqref{e-flow}. 
Thus, $\sR$ is defined as the intersection of half-spaces and so, it is closed
and convex.

The same argument shows that if for all $i \in [n]$, $\mathcal{P}_i$ is a
polytope, so is $\mathcal{R}$.
\end{proof}

We next list some concrete examples of such inequalities.


\paragraph{Skeleton constraints}
 Imagine that a~designer gave a~\NEW{skeleton} or \NEW{template} for page $i$. The latter may include a~collection of mandatory sites to be pointed by page $i$. We shall abstract the skeleton by representing it by a~fixed probability vector $q\in \Sigma_n$, giving the transition probabilities if
no further hyperlinks are added. Assume now that the webmaster is allowed to modify the page for optimization purposes, as long as the hyperlinks she adds
do not overtake the initial content of the web site. This can be modeled
by requiring that no hyperlink included in the skeleton looses
a proportion of its weight greater than $\mu$.
Such constraints can be written as 
$ 
 P_{i,j} \geq \alpha (1-\mu) q_j + (1-\alpha) z_j, \; \forall j\in[n]
$. 





\paragraph{Linear coupling constraints} \label{para:linCoupl}
Constraints like the presence of specific outlinks \emph{somewhere} on the pages of  the website are non-local. Such constraints cannot be written
simply in terms of the stochastic matrix $P$ (because adding conditional probabilities relative to different pages makes little sense) but they can be written
linearly in terms of the occupation measure $\rho$,
$ 
\sum_{i,j\in[n]} a_{i,j} \rho_{i,j} \leq b
$, 
 where the coefficients $a_{i,j}$ and $b$ are given. 

These constraints include for instance {\em coupling conditional probability constraints},
which can be written as:
$\sum_{i \in I,j \in J} \rho_{i,j} \geq b \sum_{i \in I,k \in [n]} \rho_{i,k}$.
This means that the probability for the random surfer to move to set $J$,
given that he is now in set $I$, should not be smaller than $b$.



\paragraph{Combinatorial constraints} 
In the discrete problem, one may wish to set combinatorial constraints 
like demanding the existence of a~path between two pages or sets of pages~\cite{NinKer-PRopt}, 
 setting mutual exclusion between two hyperlinks~\cite{Blondel-PRopt} or 
limiting the number of hyperlinks~\cite{Blondel-PRopt}.  
Such constraints may lead to harder combinatorial problems, the solution
of which is however made easier by the polynomial-time solvability
of a relaxed continuous problem
(Section~\ref{sec:expCoupl}). 

\section{The polytope of uniform transition measures} \label{sec:polytope}

In this section, we show that the polytope of
uniform transition measures admits a concise representation
(Theorem~\ref{thm:nonEmptySet}).
The vertices of this polytope represent the action space of 
the Discrete PageRank Optimization problem~\eqref{eqn:idealDisc}. 
Theorem~\ref{thm:nonEmptySet} is a~key ingredient of the proof
of the polynomial time character of this problem 
which will be given in the next section.

We consider a~given page $i$ and we study the set of admissible transition probabilities
from page~$i$. With uniform transitions, this is a~discrete set
that we denote $\mathcal{D}_i$. 
For clarity of the explanation, we will write $x_j$ instead of $S_{i,j}$ and write
the proofs in the case $\alpha = 1$.
To get back to $\alpha < 1$, we use the relation $P_{i,j} = \alpha S_{i,j} + (1-\alpha) z_j$ (see Remark~\ref{rem:polytopeWithDamping} at the end of this section).

We partition the set of links from page~$i$ as the set of obligatory links $\mathcal{O}_i$,
the set of prohibited links $\mathcal{I}_i$ and the set of facultative links $\mathcal{F}_i$. 
Then, depending on the presence of obligatory links,
\ifthenelse{\boolean{arxiv}}{
\begin{multline} \label{eq:discSet}
\mathcal{D}_i= \{ q \in \Sigma^n \; | \; \mathcal{O}_i \subseteq \supp(q) \subseteq \mathcal{O}_i \cup \mathcal{F}_i, \; \\
q \text{ uniform probability measure on its support} \}
\end{multline}
}
{
\begin{equation} \label{eq:discSet}
\mathcal{D}_i= \{ q \in \Sigma^n \; | \; \mathcal{O}_i \subseteq \supp(q) \subseteq \mathcal{O}_i \cup \mathcal{F}_i, \;
q \text{ uniform probability measure on its support} \}
\end{equation}
}
or if $\mathcal{O}_i = \emptyset$, it is possible to have no link at all and then to teleport with probability vector~$Z$:
\ifthenelse{\boolean{arxiv}}{
\begin{multline*}
 \mathcal{D}_i= \{ q \in \Sigma^n \; | \; \supp(q) \subseteq \mathcal{F}_i, \; \\
q \text{ uniform probability measure on its support} \}\cup \{Z\} \enspace .
\end{multline*}
}
{$ 
\mathcal{D}_i= \{ q \in \Sigma^n \; | \; \supp(q) \subseteq \mathcal{F}_i, \;
q \text{ uniform probability measure on its support} \}\cup \{Z\}.
$ 
}

We study the polytope $\co(\mathcal{D}_i)$, 
 the convex hull of the discrete set $\mathcal{D}_i$.
Although it 
 is defined as the convex hull of an~exponential number of points, 
we show that it has a~concise representation.


\begin{thm} \label{thm:nonEmptySet}
 If page $i$ has at least one obligatory link, 
then the convex hull of the admissible discrete transition probabilities from page~$i$, $\co(\mathcal{D}_i)$, 
is the projective transformation of a~hypercube of dimension $\abs{\mathcal{F}_i}$
and, for any 
 choice of $j_0 \in \mathcal{O}_i$, it coincides with the polytope defined by the following set of inequalities: 
\begin{subequations} \label{eqn:notEmpty}
\begin{alignat}{4}
&\forall j \in \mathcal{I}_i \; , \;& &x_j =0 
& \qquad &\forall j \in \mathcal{F}_i \;, \;& &x_j \leq x_{j_0} \label{eqn:superior} \\
&\forall j \in \mathcal{O}_i \setminus \{j_0\} \;,\;& & x_j = x_{j_0}\ 
& \qquad &\forall j \in \mathcal{F}_i \;, \;& &x_j \geq 0 \label{eqn:inferior} \\
& & & &\sum_{j \in [n]} x_j =1 \label{eqn:simp1} & & &
\end{alignat}
\end{subequations}
\end{thm}

\begin{proof}
Let $\mathcal{S}_i$ be the polytope defined by Inequalities~\eqref{eqn:notEmpty}. 

$(\mathcal{D}_i \subseteq \mathcal{S}_i)$:
Let $q$ a~probability vector in $\mathcal{D}_i$: $q$ is 
a uniform probability measure on its support and
$\mathcal{O}_i \subseteq \supp(q) \subseteq \mathcal{O}_i \cup \mathcal{F}_i$.
As for all $j$ in $\mathcal{F}_i$, $q_j \leq \frac{1}{\abs{\supp(q)}} = q_{j_0}$, $q$ verifies the equalities.

$(\extr(\mathcal{S}_i) \subseteq \mathcal{D}_i)$: 
 Let us consider an~extreme point $x$ of $\mathcal{S}_i$. Inequalities~\eqref{eqn:inferior} 
and~\eqref{eqn:superior} cannot be saturated together
at a~given coordinate $j \in \mathcal{F}_i$ because, if it were the case,
then we would have $x_{j_0} = 0$ and thus $x=0$, which contradicts $\sum_{j \in [n]} x_j =1$. 

We have $1+ \abs{\mathcal{I}_i} + \abs{\mathcal{O}_i}-1$ independent equalities so the polytope
is of dimension $\abs{\mathcal{F}_i}$. To be an~extreme point, $x$ must thus saturate
$\abs{\mathcal{F}_i}$ inequalities. At every $j$ in $\mathcal{F}_i$, Inequalities~\eqref{eqn:inferior}
 and~\eqref{eqn:superior} cannot be saturated simultaneously (see the previous paragraph),
so the only way to saturate $\abs{\mathcal{F}_i}$ inequalities is to saturate
one of~\eqref{eqn:inferior} or~\eqref{eqn:superior} at every $j$ in $\mathcal{F}_i$.
Finally, $x$ can only take two distinct values, which are $0$ and $x_{j_0} = \frac{1}{\abs{\supp(x)}}$:
it~is a~uniform probability on it support.

We then show that $\mathcal{S}_i$ is the projective transformation
(\cite{Ziegler-polytopes}, Section~2.6 for more background)
of the~hypercube $H$ defined by %
\ifthenelse{\boolean{arxiv}}{the following set of inequalities:
\[\{ \forall j \in \mathcal{I}_i, X_j = 0 \; ;\;
\forall j \in \mathcal{O}_i, X_j =1 \; ; \; \forall j \in \mathcal{F}_i, 
0 \leq X_j \leq 1 \} \enspace .
\]}
{$ 
\{ \forall j \in \mathcal{I}_i, X_j = 0 \; ;\;
\forall j \in \mathcal{O}_i, X_j =1 \; ; \; \forall j \in \mathcal{F}_i, 
0 \leq X_j \leq 1 \}
$.} 
As\ifthenelse{\boolean{arxiv}}{}{ soon as} $\mathcal{O}_i \not = \emptyset$, $H$ is embedded in the affine hyperplane 
$\{X \in \mathbb{R}^n | X_{j_0} = 1 \}$. We can then construct the homogenization of $H$,
$\mathrm{homog}(H)$, which is the pointed cone with base~$H$ (see~\cite{Ziegler-polytopes} for more details). Finally $\mathcal{S}_i$ is
the cross-section of $\mathrm{homog}(H)$ with the hyperplane $\{x \in \mathbb{R}^n | \sum_{j \in [n]} x_j = 1 \}$.
\end{proof}
The result of the theorem implies in particular that $\co(\mathcal{D}_i)$ is combinatorially equivalent to a hypercube, ie. that their face lattices are isomorphic~\cite{Ziegler-polytopes}.

\ifthenelse{\boolean{arxiv}}
{
\begin{figure}[thbp]
\centering
 \includegraphics[scale=0.33]{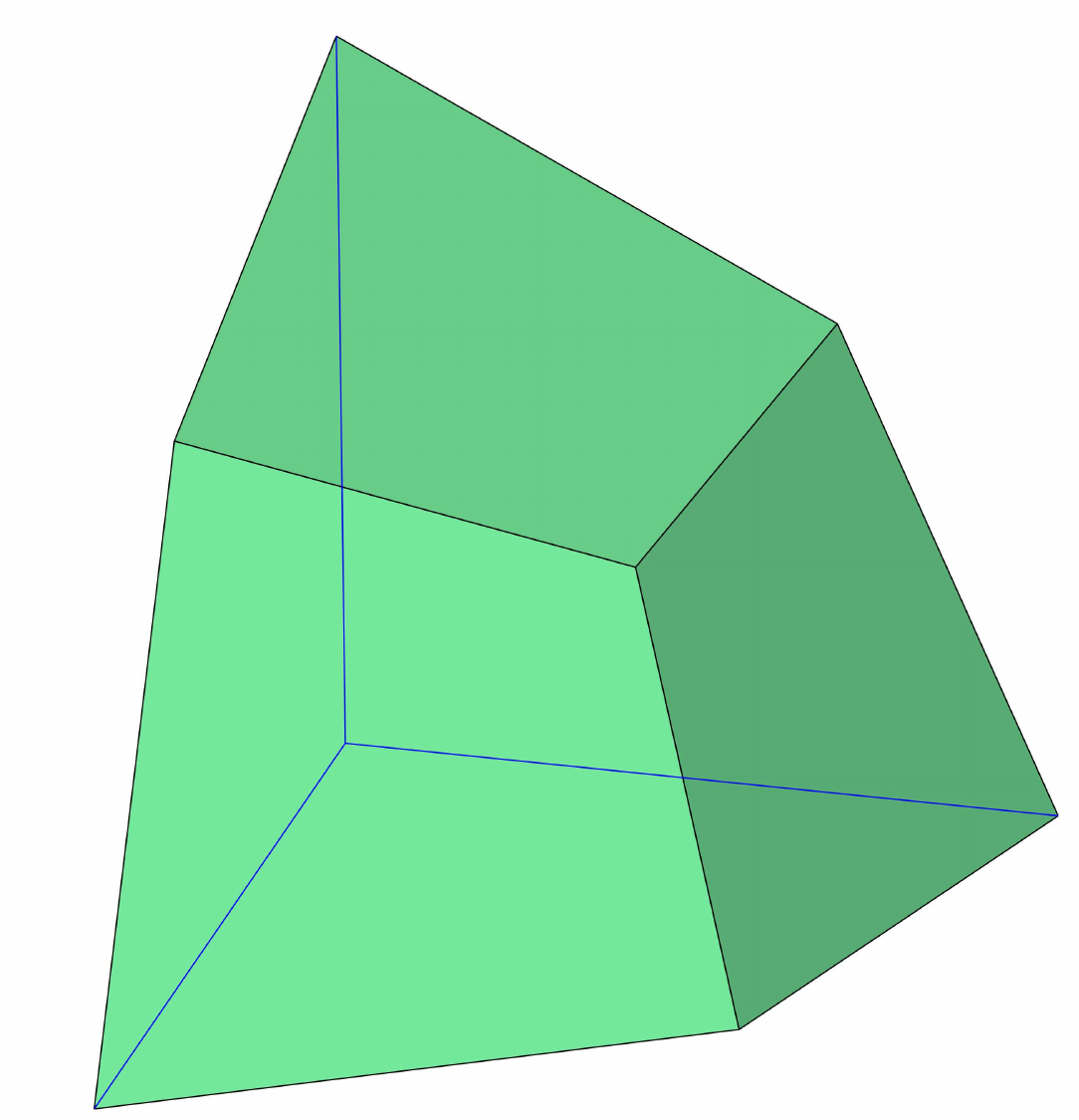}
\caption{Projection of the polytope of uniform transition measures with one obligatory link ($\abs{\mathcal{O}_i}=1$) and three facultative links ($\abs{\mathcal{F}_i}=3$).}
\label{draw:hypercube}
\end{figure}
}
{
}

The next result concerns the case in which a page may have no outlink:
it is necessary to consider this special case because then the websurfer teleports
with probability $Z_i$ to page $i$.
\begin{prop} \label{prop:emptySet}
 If page $i$ has no obligatory link 
and if 
there exists $k \in \mathcal{I}_i$ such that $Z_k > 0$,
 then $\co(\mathcal{D}_i)$ 
is a~simplex of dimension $\abs{\mathcal{F}_i}$ defined by the following set of inequalities: 
\begin{subequations} \label{eqn:empty}
\begin{align} 
& \sum_{j \in [n]} x_j =1 \;, & x_k \geq 0 \label{eqn:base} \\ 
&\forall j \in \mathcal{I}_i \setminus \{k\} \; ,\quad x_j = \frac{Z_j}{Z_k} x_k \;,  & 
\forall j \in \mathcal{F}_i \;,\quad& x_j \geq \frac{Z_j}{Z_k} x_k \label{eqn:lateralFaces} 
\end{align}
\end{subequations}
\end{prop}
\begin{proof}
%
%
%
%
%
The proof follows the same sequence of arguments as the proof of Theorem~\ref{thm:nonEmptySet}.
We just need to adapt it to Inequalities~\eqref{eqn:empty}.
\end{proof}

\begin{prop} \label{prop:emptySet2}
 If page $i$ has no obligatory link 
and if 
for all $k \in \mathcal{I}_i$, $Z_k = 0$,
 then $\co(\mathcal{D}_i)$ 
 is the usual simplex of dimension $\abs{\mathcal{F}_i}-1$ with $x_j = 0$, $\forall j \in \mathcal{I}_i$.
\end{prop}
\begin{proof}
The extreme points of this simplex are clearly admissible discrete transition probabilities
and the polytope contains every admissible discrete transition probabilities.
\end{proof}

\begin{remark}
When there is no obligatory link, 
 most of the admissible discrete transition probabilities
are not extreme points of the polytope.
\end{remark}
\begin{remark} \label{rem:polytopeWithDamping}
 If we want to work with $\mathcal{P}_i$, the polytope of transition probabilities 
with damping factor $\alpha$, we only need\ifthenelse{\boolean{arxiv}}{}{ to use} the relation $\mathcal{P}_i = \alpha \mathcal{S}_i + (1-\alpha)z$
to get the actual inequalities.
For instance, $x_j = x_{j_0}$ remains but $x_j \geq 0$ becomes $x_j \geq (1-\alpha)z_j$. 
\end{remark}

\section{Solving the PageRank Optimization Problem with local constraints}
\label{sec:solvelocal}

\subsection{Reduction of the PageRank Optimization Problem with local constraints to Ergodic Control}
\label{sec:progdyn}

We next show 
that the continuous and discrete versions of the PageRank optimization 
reduce to ergodic control problems in which the action sets are defined as extreme points 
of concisely described polyhedra.

A finite {\em Markov decision process} is a $4$-uple $(I,(A_i)_{i\in I},p,r)$
where $I$ is a finite set called the {\em state space}; for all $i \in I$, $A_i$ is the finite set of {\em admissible actions} in state $i$;
$p: I\times \cup_{i\in I} (\{i\} \times A_i) \to \mathbb{R}_+$ is the {\em transition law}, so that $p(j|i,a)$ is the probability to go to state $j$ form state $i$ when
action $a \in A_i$ is selected; and $r: \cup_{i\in I}( \{i\} \times A_i) \to \mathbb{R}$ is the {\em reward function}, so that $r(i,a)$ is the instantaneous reward
when action $a$ is selected in state $i$.

Let $X_t\in I$ denote the state of the system at the discrete time $t\geq 0$.
A {\em deterministic control strategy} $\nu$ is a~sequence of actions $(\nu_t)_{t\geq 0}$
such that for all $t \geq 0$, $\nu_t$ is a~function of the history $h_t=(X_0, \nu_0, \ldots, X_{t-1}, \nu_{t-1}, X_t)$ and $\nu_t \in A_{X_t}$.
Of course, $\mathbb{P}( X_{t+1}=j | X_t , \nu_t ) = p(j|X_t, \nu_t) , \forall j \in [n], \forall t \geq 0$.
More generally, we may consider {\em randomized} strategies $\nu$ where $\nu_t$ is a probability measure on $A_{X_t}$.
A strategy $\nu$ is {\em stationary} (feedback) if there exists a function $\bar{\nu}$
such that for all $t\geq 0$, $\nu_t(h_t) = \bar{\nu}(X_t)$.

Given an initial distribution $\mu$ representing the law of $X_0$, the average cost infinite horizon Markov decision problem, also called {\em ergodic control problem},  consists in maximizing 
\begin{equation}\label{eqn:ergo}
\liminf_{T \rightarrow +\infty} \frac{1}{T}
 \mathbb{E}(\sum_{t=0}^{T-1} r(X_t,  \nu_t))
\end{equation}
where the maximum is taken 
over the set of randomized control strategies $\nu$.
Indeed, the supremum is the same if it is taken only over the set 
of randomized (or even deterministic) stationary feedback strategies (Theorem~9.1.8 in~\cite{puterman-mdp} for instance). 

A Markov decision process is {\em unichain} if the transition matrix corresponding to every
stationary policy has a single recurrent class. Otherwise it is {\em multichain}. 
When the problem is unichain, its value does not depend on the initial distribution whereas
when it is not, one may consider a vector $(g_i)_{i \in I}$ where $g_i$ represents the value of the problem~\eqref{eqn:ergo} when starting from state $i$. 


\begin{prop}\label{prop-local}
If there are only local constraints, ie. $\mathcal{P}=\prod_{i \in [n]} \mathcal{P}_i$, if for all $i\in [n]$, $\mathcal{P}_i$ is a polytope and if the utility function is an~income proportional
to the frequency of clicks~(\ref{eqn:income}), then the
continuous PageRank Optimization problem~\eqref{eqn:ideal}
is equivalent to the unichain ergodic control problem with finite state $[n]$, finite action set $\extr(\mathcal{P}_i)$ in every state $i$, transition probabilities $p(j|i,a)=a_j$ and
rewards $r(i,a)=\sum_{j\in [n]} r_{i,j} a_j$.
 \end{prop}
%
\begin{proof}
As $\alpha<1$, $a \in \mathcal{P}_i$ implies $a_k>0$ for all $k$. Thus the problem defined in the proposition is unichain.
Randomized stationary strategies 
are of the form $\nu_t=\bar{\nu}(X_t)$ for some
function $\bar{\nu}$ sending $i\in[n]$ to some element of $\mathcal{P}_i=\co(\extr(\mathcal{P}_i))$.
To such a~strategy is associated a~transition matrix $P$ of the websurfer, 
obtained by taking $P_{i,\cdot}=\bar{\nu}(i)$ and vice versa.
Thus, the admissible transition matrices of the websurfer
are admissible stationary feedback strategies.

Moreover, the ergodic theorem for Markov chains shows that when such
a strategy is applied,
\ifthenelse{\boolean{arxiv}}
{\begin{align*}
\lim_{T\to\infty} \frac{1}{T} \mathbb{E}(\sum_{t=0}^{T-1} r(X_t,\nu_t)) = \lim_{T\to\infty} \frac{1}{T} \mathbb{E}(\sum_{t=0}^{T-1} \sum_{j\in [n]} r_{X_t,j} \bar{\nu}_j(X_t)) =
\sum_{i,j\in[n]} \pi_i P_{i,j}r_{i,j}
 \end{align*}}
{$ 
\lim_{T\to\infty} \frac{1}{T} \mathbb{E}(\sum_{t=0}^{T-1} r(X_t,\nu_t)) = \lim_{T\to\infty} \frac{1}{T} \mathbb{E}(\sum_{t=0}^{T-1} \sum_{j\in [n]} r_{X_t,j} \bar{\nu}_j(X_t)) =
\sum_{i,j\in[n]} \pi_i P_{i,j}r_{i,j}
$} 
and so, the objective function of the ergodic control problem is precisely
the total income.
\end{proof}

\begin{prop} \label{prop:progdynprinc}
The following dynamic programming equation:
\begin{equation}\label{eqn:progdyn}
  w_i+\psi = \max_{\nu \in \mathcal{P}_i} \nu (r_{i,\cdot} + w) \enspace,
\qquad \forall i\in [n]
\end{equation}
has a~solution $w \in \mathbb{R}^n$ and $\psi \in \mathbb{R}$. The constant $\psi$ is unique and is the value of the
PageRank Optimization problem~(\ref{eqn:ideal}). An optimal strategy is obtained by selecting
for each state $i$ a~maximizing $\nu \in \mathcal{P}_i$ in~\eqref{eqn:progdyn}.
The function $w$ is often called the \NEW{bias} or the \NEW{potential}.
\end{prop}
\begin{proof}
Theorem~8.4.3 in~\cite{puterman-mdp}
 applied to the unichain ergodic control problem of Proposition~\ref{prop-local}
implies the result of the proposition but with 
$\mathcal{P}_i$ replaced by 
$\extr(\mathcal{P}_i)$. 
But as
the expression which is maximized is affine,
using $\mathcal{P}_i$ or $\extr(\mathcal{P}_i)$ yields the same solution.
\end{proof}


\begin{thm} \label{thm:reducDiscCont}
 The discrete PageRank Optimization problem~\eqref{eqn:idealDisc} is equivalent
to a~continuous PageRank Optimization problem~\eqref{eqn:ideal} in which
the action set $\mathcal{P}_i$ is defined by one of the polytopes described in Theorem~\ref{thm:nonEmptySet}
or Proposition~\ref{prop:emptySet} or~\ref{prop:emptySet2}, depending on the presence of
obligatory links.
\end{thm}
\begin{proof}
Arguing as in the proof of Proposition~\ref{prop-local}, we get that the 
discrete PageRank Optimization problem~\eqref{eqn:idealDisc} is equivalent to an ergodic
control problem with state space $[n]$, in which the action set in 
state $i$ is the discrete set $\mathcal{D}_i$ defined in~\eqref{eq:discSet}, and the rewards and transition
probabilities are as in this proposition. 
The optimal solutions of the discrete PageRank Optimization problem
coincide with the optimal stationary deterministic strategies.
The analog of Equation~\eqref{eqn:progdyn} is now
\begin{equation}
  w_i+\psi = \max_{\nu \in \co(\mathcal{D}_i)} \nu (r_{i,\cdot} + w) 
\label{e-w-psi}
\end{equation}
where $\co(\mathcal{D}_i)$ is the convex hull of the set $\mathcal{D}_i$, i.e the polytope
described in either Theorem~\ref{thm:nonEmptySet}
or Proposition~\ref{prop:emptySet} or~\ref{prop:emptySet2}. The polytope $\co(\mathcal{D}_i)$ gives the transition laws in state $i$ corresponding to randomized strategies in the former problem. Hence, the control problems in which the actions sets are $\mathcal{D}_i$ or $\co(\mathcal{D}_i)$ have the same value. 
Moreover, an optimal strategy of the problem with the latter
set of actions can be found by solving~\eqref{e-w-psi} and selecting
a maximizing action $\nu$ in~\eqref{e-w-psi}. Such an action may
always be chosen in the set of extreme points of $\co(\mathcal{D}_i)$
and these extreme points belong to $\mathcal{D}_i$ (beware however
that some points of $\mathcal{D}_i$ may be not extreme). 
\end{proof}

\subsection{Polynomial time solvability of well-described Markov decision problems}

We have reduced the discrete and continuous PageRank Optimization problems
to ergodic control problems in which the action sets are {\em implicitly defined} as the sets of extreme points of polytopes. Theorem~1 in~\cite{Papadimitriou-complMDP} states that
the ergodic control problem is solvable in polynomial time.
However, in this result, the action sets are defined {\em explicitly},
whereas polynomial means, as usual, polynomial in the input length (number of bits of the input). Since the input includes
the description of the actions sets, 
the input length is always larger than the sum of the cardinalities of the
action sets. 
Hence, this result only leads to
an exponential bound in our case (Remark~\ref{rem:combinatorialPb}).

However, we next establish a general result, Theorem~\ref{thm:polytime+} below,
showing that the polynomial time solvability of ergodic control problems subsists when the action sets are implicitly defined.
This is based on the combinatorial developments of the theory of Khachiyan's ellipsoid method, by Groetschel, Lov\'asz and Schrijver~\cite{Lovasz-geomAlgo}.
We refer the reader to the latter monograph for more background on
the notions of strong separation oracles and well described polyhedra.

\begin{defn}[Def.\ 6.2.2 of \cite{Lovasz-geomAlgo}]  \label{defn:welldescribed}
We say that a polyhedron $\mathcal{B}$ has {\em facet-complexity} at most~$\phi$ if there exists
a~system of inequalities with rational coefficients that has solution set~$\mathcal{B}$ and
such that the encoding length of each inequality of the system (the sum of the number of bits of the rational numbers appearing as coefficients
in this inequality) is at most~$\phi$. 

A {\em well-described polyhedron} is a~triple $(\mathcal{B};n,\phi)$ where $\mathcal{B} \in \mathbb{R}^n$
is a~polyhedron with facet-complexity at most $\phi$. The {\em encoding length} of $\mathcal{B}$ is by definition $n+\phi$.
\end{defn}
\begin{defn}[Problem (2.1.4) of \cite{Lovasz-geomAlgo}]
A~{\em strong separation oracle} for a set $K$ is an~algorithm that solves the
following problem: given a~vector~$y$, 
decide whether $y \in K$ or not and if not, find a~hyperplane
that separates $y$ from~$K$; 
i.e.,
find a~vector~$c$ such
that $c^T y > \max \{c^T x, x\in K\}$.
%
\end{defn}

Inspired by Definition~\ref{defn:welldescribed}, we introduce the following notion.
\begin{defn} \label{defn:welldescribedMDP}
A finite Markov decision process $(I, (A_i)_{i\in I}, p, r)$ is {\em well-described} 
if for every state $i \in I$, we have $A_i \subset \mathbb{R}^{L_i}$
for some $L_i \in \mathbb{N}$, 
if there exists $\phi \in \mathbb{N}$
such that  the convex hull of every action set $A_i$ is a~well-described polyhedron $(\mathcal{B}_i; L_i, \phi)$
with a~polynomial time strong separation oracle, and if
the rewards and transition probabilities satisfy
$r(i,a)=\sum_{l \in [L_i]} a_l R_{i}^l$
and
$p( j | i , a) = \sum_{l \in [L_i]} a_l Q_{i,j}^l$, $\forall i,j \in I$, 
$\forall a\in  A_i$,
where $R_{i}^l$ and $Q_{i,j}^l$
are given rational numbers,
for $i,j\in I$ and $l\in [L_i]$.

The {\em encoding length} of a well-described Markov decision process is by definition the sum of the encoding lengths of the rational numbers $Q_{i,j}^l$ and $R_{i}^l$ and of the well-described polyhedra~$\mathcal{B}_i$.
\end{defn}

The situation in which the action spaces are given as usual in extension (by listing the actions) corresponds to the case in which $A_i$ is
the set of extreme points of a simplex $\Sigma_{L_i}$. 
The interest of Definition~\ref{defn:welldescribedMDP} is that it applies
to more general situations in which the actions are not listed, but
given implicitly by a computer program
deciding whether a given element of $\mathbb{R}^{L_i}$ is an admissible 
action in state $i$ (the separation oracle).
An example of such a separation oracle stems from
Theorem~\ref{thm:nonEmptySet}: here, a potential (randomized) action is an element of $\mathbb{R}^n$, and to check whether it is admissible, it suffices to check whether one
of the inequalities in~\eqref{eqn:notEmpty} is not satisfied.

\begin{thm} \label{thm:polytime+}
The average cost infinite horizon problem for a well-described (multichain) Markov decision process can be solved in a time
polynomial in the input length.
\end{thm}


\begin{proof}
We shall use the notations of Definition~\ref{defn:welldescribedMDP}. Consider the polyhedron $\mathcal{G}$ consisting of the couples of vectors 
$(v,g)\in \R^I \times \R^I$ satisfying the constraints
\begin{equation} \label{exponentialLP}
\begin{split}
g_i \geq \sum_{j \in I} \sum_{l \in [L_i]} a_l Q_{i,j}^l g_j \;, &\quad \forall i \in I, a \in A_i \\
v_i + g_i \geq \sum_{l \in [L_i]} a_l R_i^l +  \sum_{j \in I}  \sum_{l \in [L_i]} &a_l Q_{i,j}^l  v_j \;, \quad \forall i \in I, a \in A_i \enspace.
\end{split}
\end{equation}
Theorem~9.3.8 in~\cite{puterman-mdp} implies that
the average cost problem
reduces to minimizing the linear form $(v,g) \mapsto \sum_{j\in I} g_j$
over $\mathcal{G}$.
Every optimal solution $(v,g)$ of this linear program is such that $g_j$ is the 
optimal mean payment per time unit starting from state~$j$. 
We recover optimal strategies of the ergodic problem through dual optimal solution of the linear program.

By Theorem~6.4.9 in~\cite{Lovasz-geomAlgo}, we know that a linear program
over a well-described polyhedron 
with a polynomial time strong separation oracle
is polynomial time solvable. 
Moreover, Theorem~6.5.14 in~\cite{Lovasz-geomAlgo}
asserts that we can find a dual optimal solution in polynomial time.

Let us construct such an oracle for $\mathcal{G}$.
Given a~point $(g,v) \in \mathbb{Q}^n \times \mathbb{Q}^n$,
compute for all $i \in I$: 
$\max_{a \in \co(A_i)} \sum_{l \in [L_i]} a_l (\sum_{j \in I} Q_{i,j}^l g_j) - g_i$ and
$\max_{a \in \co(A_i)} \sum_{l \in [L_i]} a_l (R_i^l +  \sum_{j \in I} Q_{i,j}^l v_j) - v_i - g_i $.
Those problems are linear problems such that, by hypothesis, we have a~polynomial time strong separation
oracle for each of the well-described polyhedral admissible sets $\mathcal{B}_i=\co(A_i)$.
Thus they are polynomial time solvable.
If the $2n$ linear programs return a~nonpositive value,
then this means that $(g,v)$ is an~admissible point of~\eqref{exponentialLP}.
Otherwise, the solution $a$ of any of those linear programs that have a negative
value yields a strict inequality
$g_i < \sum_{j \in I} \sum_{l \in [L_i]} a_l Q_{i,j}^l g_j$ or $v_i + g_i < \sum_{l \in [L_i]} a_l R_i^l +  \sum_{j \in I}  \sum_{l \in [L_i]} a_l Q_{i,j}^l  v_j$.
In both cases, the corresponding inequality determines a~separating hyperplane.

To conclude the proof, it remains to check that the facet complexity
of the polyhedron $\mathcal{G}$
is polynomially bounded in the encoding lengths of the polyhedra $\mathcal{B}_i$ and the rationals $R_i^l$ and~$Q_{i,j}^l$. Since the $a_l$'s appear linearly in the constraints~\eqref{exponentialLP}, these constraints hold for all $a\in A_i$ if and only if they
hold for all $a\in \mathcal{B}_i$ or equivalently, for all extreme points of~$\mathcal{B}_i$. 
The result follows from
Lemma~6.2.4 in~\cite{Lovasz-geomAlgo}, which states that the encoding length of any extreme point
of a~well-described polyhedron 
is polynomially bounded in the encoding of the polyhedron.
\end{proof}
\begin{remark}
 This argument also shows that the discounted problem is polynomial time solvable.
\end{remark}

 As a~consequence of Theorems~\ref{thm:reducDiscCont} and~\ref{thm:polytime+}, we get
\begin{thm} \label{thm:polytime}
If there are only local constraints, if the utility function is a~rational total income utility~\eqref{eqn:income}
and if the teleportation vector and damping factor are rational,
then 
the discrete problem~\eqref{eqn:idealDisc} can be solved in polynomial time
and the continuous problem~\eqref{eqn:ideal} with well-described action sets (Definition~\ref{defn:welldescribed}) can also
be solved in polynomial time. 
\end{thm}
\begin{proof}
Thanks to Theorem~\ref{thm:reducDiscCont}, solving the continuous PageRank
Optimization problem also solves the discrete PageRank Optimization problem.
In addition, the coefficients appearing in the description of the facets of the polytopes
of uniform transition measures
are either $1$, $z_j$ or $\alpha$ and there are at most two
terms by inequality (cf Section~\ref{sec:polytope}). This implies that these polytopes are well-described
with an encoding length polynomial in the length of the input.
Note also that we can find in polynomial time a vertex optimal solution of a linear program as soon
as its feasible set is a polytope as it is the case here
(Lemma~6.5.1 in~\cite{Lovasz-geomAlgo}).

By Proposition~\ref{prop-local}, the ergodic control problem associated to a~continuous PageRank
Optimization problem with well-described action sets satisfies the conditions of Theorem~\ref{thm:polytime+}
with $I=[n]$, $L_i=[n]$, $Q_{i,j}^l = \delta_{jl}$ and $R_i^l=r_{i,l}$ for $i,j \in [n], l \in L_i$.
Thus it is polynomial time solvable.
\end{proof}
Theorem~\ref{thm:polytime+} is mostly of theoretical interest, since
its proof is based on the ellipsoid algorithm, which is slow. We however
give in Section~\ref{subsec-value} a fast scalable algorithm for the present problem.

\begin{example}
Consider again the graph from Figure~\ref{graph:base},
and let us optimize the sum of the PageRank scores of the pages of the site
(colored). Assume that there are only local skeleton constraints (see Section~\ref{sec:cont}): each page can change up to 20 \% of the initial transition probabilities. 
The result is represented in Figure~\ref{graph:local}. 

\begin{figure}[thbp]
\centering
\ifthenelse{\boolean{notikz}}
{
\includegraphics[width=23em]{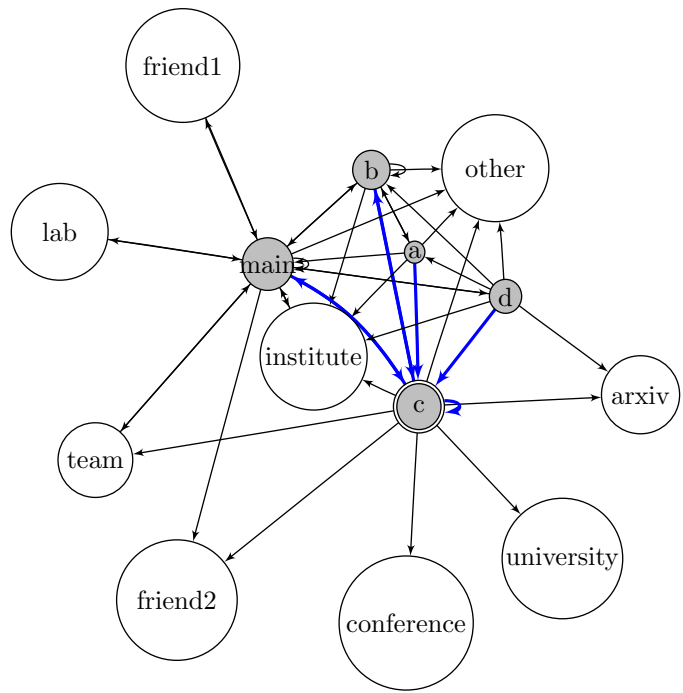}
}
{
\small{
\begin{tikzpicture}[>=latex', join=bevel, scale=0.32]
\pgfsetlinewidth{0.5bp}
\pgfsetcolor{black}
  \pgfsetcolor{blue}
  \draw [->,very thick] (429bp,592bp) .. controls (439bp,545bp) and (461bp,435bp)  .. (475bp,365bp);
  \pgfsetcolor{black}
  \draw [->] (599bp,454bp) .. controls (621bp,437bp) and (664bp,406bp)  .. (704bp,376bp);
  \draw [->] (333bp,510bp) .. controls (343bp,510bp) and (351bp,508bp)  .. (351bp,503bp) .. controls (351bp,500bp) and (348bp,498bp)  .. (333bp,496bp);
  \draw [->] (464bp,516bp) .. controls (439bp,514bp) and (382bp,510bp)  .. (333bp,505bp);
  \draw [->] (333bp,499bp) .. controls (387bp,492bp) and (500bp,476bp)  .. (564bp,468bp);
  \pgfsetcolor{blue}
  \draw [->,very thick] (475bp,365bp) .. controls (464bp,418bp) and (442bp,528bp)  .. (429bp,592bp);
  \pgfsetcolor{black}
  \draw [->] (129bp,305bp) .. controls (169bp,350bp) and (238bp,429bp)  .. (283bp,480bp);
  \draw [->] (291bp,531bp) .. controls (277bp,564bp) and (252bp,620bp)  .. (229bp,675bp);
  \draw [->] (467bp,508bp) .. controls (454bp,495bp) and (430bp,471bp)  .. (401bp,440bp);
  \draw [->] (569bp,478bp) .. controls (542bp,504bp) and (483bp,560bp)  .. (442bp,598bp);
  \draw [->] (502bp,313bp) .. controls (525bp,289bp) and (562bp,249bp)  .. (601bp,208bp);
  \pgfsetcolor{blue}
  \draw [->,very thick] (330bp,489bp) .. controls (351bp,478bp) and (380bp,461bp)  .. (402bp,440bp) .. controls (425bp,419bp) and (445bp,392bp)  .. (465bp,361bp);
  \pgfsetcolor{black}
  \draw [->] (409bp,599bp) .. controls (389bp,582bp) and (356bp,552bp)  .. (325bp,523bp);
  \draw [->] (511bp,337bp) .. controls (554bp,339bp) and (632bp,343bp)  .. (696bp,346bp);
  \draw [->] (331bp,515bp) .. controls (372bp,532bp) and (448bp,564bp)  .. (513bp,591bp);
  \draw [->] (418bp,593bp) .. controls (409bp,563bp) and (392bp,508bp)  .. (376bp,454bp);
  \draw [->] (295bp,474bp) .. controls (279bp,413bp) and (241bp,275bp)  .. (215bp,176bp);
  \pgfsetcolor{blue}
  \draw [->,very thick] (571bp,450bp) .. controls (555bp,429bp) and (527bp,394bp)  .. (500bp,359bp);
  \draw [->,very thick] (476bp,504bp) .. controls (477bp,478bp) and (479bp,417bp)  .. (480bp,366bp);
  \pgfsetcolor{black}
  \draw [->] (325bp,523bp) .. controls (346bp,543bp) and (379bp,572bp)  .. (408bp,599bp);
  \draw [->] (447bp,614bp) .. controls (461bp,614bp) and (479bp,615bp)  .. (508bp,615bp);
  \draw [->] (230bp,675bp) .. controls (248bp,633bp) and (272bp,577bp)  .. (291bp,531bp);
  \draw [->] (436bp,594bp) .. controls (444bp,577bp) and (457bp,554bp)  .. (470bp,528bp);
  \draw [->] (454bp,348bp) .. controls (445bp,352bp) and (434bp,357bp)  .. (414bp,367bp);
  \draw [->] (115bp,532bp) .. controls (160bp,525bp) and (223bp,516bp)  .. (273bp,508bp);
  \pgfsetcolor{blue}
  \draw [->,very thick] (465bp,361bp) .. controls (451bp,384bp) and (428bp,416bp)  .. (402bp,440bp) .. controls (383bp,458bp) and (359bp,473bp)  .. (330bp,489bp);
  \pgfsetcolor{black}
  \draw [->] (490bp,364bp) .. controls (503bp,407bp) and (530bp,488bp)  .. (552bp,556bp);
  \draw [->] (470bp,528bp) .. controls (463bp,542bp) and (450bp,566bp)  .. (436bp,594bp);
  \draw [->] (581bp,485bp) .. controls (580bp,500bp) and (579bp,522bp)  .. (576bp,553bp);
  \draw [->] (565bp,459bp) .. controls (535bp,450bp) and (475bp,431bp)  .. (417bp,413bp);
  \draw [->] (273bp,508bp) .. controls (237bp,514bp) and (173bp,524bp)  .. (115bp,532bp);
  \draw [->] (457bp,316bp) .. controls (414bp,282bp) and (322bp,208bp)  .. (252bp,152bp);
  \draw [->] (446bp,621bp) .. controls (456bp,621bp) and (465bp,619bp)  .. (465bp,614bp) .. controls (465bp,611bp) and (461bp,609bp)  .. (446bp,607bp);
  \draw [->] (329bp,451bp) .. controls (327bp,456bp) and (324bp,462bp)  .. (317bp,476bp);
  \draw [->] (564bp,468bp) .. controls (519bp,474bp) and (404bp,489bp)  .. (333bp,499bp);
  \draw [->] (451bp,330bp) .. controls (387bp,319bp) and (235bp,294bp)  .. (144bp,279bp);
  \draw [->] (479bp,304bp) .. controls (477bp,271bp) and (474bp,216bp)  .. (471bp,159bp);
  \draw [->] (566bp,473bp) .. controls (547bp,482bp) and (516bp,498bp)  .. (487bp,512bp);
  \pgfsetcolor{blue}
  \draw [->,very thick] (511bp,342bp) .. controls (521bp,342bp) and (529bp,340bp)  .. (529bp,335bp) .. controls (529bp,332bp) and (526bp,330bp)  .. (511bp,328bp);
  \pgfsetcolor{black}
  \draw [->] (283bp,480bp) .. controls (248bp,441bp) and (179bp,362bp)  .. (129bp,305bp);
  \draw [->] (485bp,526bp) .. controls (494bp,535bp) and (506bp,549bp)  .. (527bp,570bp);
  \draw [->] (317bp,475bp) .. controls (320bp,471bp) and (322bp,465bp)  .. (329bp,451bp);
\begin{scope}
  \pgfsetstrokecolor{black}
  \pgfsetfillcolor{lightgray}
  \filldraw (476bp,517bp) ellipse (12bp and 13bp);
  \draw (476bp,517bp) node {\href{http://amadeus.inria.fr/gaubert/PAPERS/TU-0190}{a}};
\end{scope}
\begin{scope}
  \pgfsetstrokecolor{black}
  \draw (466bp,80bp) ellipse (79bp and 79bp);
  \draw (466bp,80bp) node {\href{http://www.siam.org/meetings/ct09}{conference}};
\end{scope}
\begin{scope}
  \pgfsetstrokecolor{black}
  \pgfsetfillcolor{lightgray}
  \filldraw (481bp,335bp) ellipse (26bp and 27bp);
  \draw (481bp,335bp) ellipse (30bp and 31bp);
  \draw (481bp,335bp) node {\href{http://amadeus.inria.fr/gaubert/programme/index.html}{c}};
\end{scope}
\begin{scope}
  \pgfsetstrokecolor{black}
  \pgfsetfillcolor{lightgray}
  \filldraw (425bp,614bp) ellipse (22bp and 23bp);
  \draw (425bp,614bp) node {\href{http://amadeus.inria.fr/gaubert/HOWARD2.html}{b}};
\end{scope}
\begin{scope}
  \pgfsetstrokecolor{black}
  \pgfsetfillcolor{lightgray}
  \filldraw (583bp,465bp) ellipse (19bp and 20bp);
  \draw (583bp,465bp) node {\href{http://amadeus.inria.fr/gaubert/papers.html}{d}};
\end{scope}
\begin{scope}
  \pgfsetstrokecolor{black}
  \draw (357bp,394bp) ellipse (63bp and 63bp);
  \draw (357bp,394bp) node {\href{http://www.inria.fr}{institute}};
\end{scope}
\begin{scope}
  \pgfsetstrokecolor{black}
  \draw (650bp,156bp) ellipse (71bp and 71bp);
  \draw (650bp,156bp) node {\href{http://www.lsv.ens-cachan.fr/Seminaires/anciens?l=fr&sem=200811041100}{university}};
\end{scope}
\begin{scope}
  \pgfsetstrokecolor{black}
  \draw (742bp,349bp) ellipse (46bp and 46bp);
  \draw (742bp,349bp) node {\href{http://dx.doi.org/10.1145/1190095.1190110}{arxiv}};
\end{scope}
\begin{scope}
  \pgfsetstrokecolor{black}
  \pgfsetfillcolor{lightgray}
  \filldraw (303bp,503bp) ellipse (30bp and 31bp);
  \draw (303bp,503bp) node {\href{http://amadeus.inria.fr/gaubert}{main}};
\end{scope}
\begin{scope}
  \pgfsetstrokecolor{black}
  \draw (58bp,541bp) ellipse (57bp and 57bp);
  \draw (58bp,541bp) node {\href{http://www.cmap.polytechnique.fr}{lab}};
\end{scope}
\begin{scope}
  \pgfsetstrokecolor{black}
  \draw (571bp,616bp) ellipse (63bp and 63bp);
  \draw (571bp,616bp) node {\href{http://users.dsic.upv.es/~sas2008/accepted_papers.html}{other}};
\end{scope}
\begin{scope}
  \pgfsetstrokecolor{black}
  \draw (100bp,272bp) ellipse (44bp and 44bp);
  \draw (100bp,272bp) node {\href{http://www.cmap.polytechnique.fr/spip.php?rubrique103}{team}};
\end{scope}
\begin{scope}
  \pgfsetstrokecolor{black}
  \draw (196bp,107bp) ellipse (71bp and 71bp);
  \draw (196bp,107bp) node {\href{http://www-rocq.inria.fr/metalau/cohen}{friend2}};
\end{scope}
\begin{scope}
  \pgfsetstrokecolor{black}
  \draw (203bp,737bp) ellipse (67bp and 67bp);
  \draw (203bp,737bp) node {\href{http://www-rocq.inria.fr/metalau/quadrat}{friend1}};
\end{scope}
\end{tikzpicture} 
}
}
\caption{Web graph of Figure~\ref{graph:base} optimized under local skeleton constraints. The optimal strategy consists in linking as much as possible to page ''c'' (actually, the page of a~lecture), up to saturating the
skeleton constraint. This page gains then a~PageRank comparable to the one of the main page.
The sum of the PageRank scores has been increased by 22.6\%.
}
\label{graph:local}
\end{figure}
\end{example}

\begin{example}
We now consider a discrete Pagerank optimization problem starting from the same graph.
We set obligatory links to be the initial links and we represent them on the adjacency matrix in Figure~\ref{graph:discrete}
by squares. Facultative links are all other possible links from controlled pages.

\begin{figure}[thbp]
\centering
\includegraphics[width=17em]{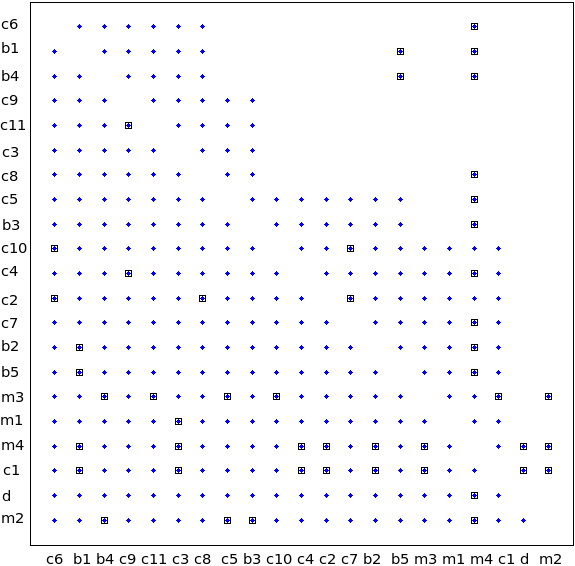}
\caption{The web graph optimized under discrete uniform transitions constraints. In this case, the optimized graph has almost all internal links (links from a controlled page to another controlled page), so, for more
readability, 
we display its adjacency matrix. 
The hyperlinks correspond to blue dots, obligatory links correspond to squares. The pages are ordered by decreasing average reward before teleportation (Section~\ref{sec:shape}). 
The optimal strategy consists in adding a~lot of internal links excluding
certain pages, as will be explained by the master Page theorem below (Theorem~\ref{thm:masterPageDisc}).}
\label{graph:discrete}
\end{figure}
\end{example}

\subsection{Optimizing the PageRank via Value iteration} \label{subsec-value} 

The PageRank optimization is likely not to be applied to the world
wide web, but rather to a~fragment of it, consisting of a~web site
(or of a~collection of web sites of a~community) and of related
sites (see Remark~\ref{rem:incompCrawling} in Section~\ref{sec:shape}) 
However, even in such simplified instances, the number of 
design variables may be large, typically between thousands and millions.
Hence, it is desirable to have scalable algorithms. 
We next describe two methods, showing that the optimization problem is computationally 
easy when there are no coupling constraints: then, optimizing
the PageRank is essentially not more expensive than computing the PageRank.

\begin{prop}  \label{prop:contrrate}
 Let $T$ be the dynamic programming operator $\R^n\to \R^n$ defined by
\begin{equation*}
T_i(w) = \max_{\nu \;\text{st} \; \alpha \nu + (1-\alpha) z \in \mathcal{P}_i} 
\alpha \nu (r_{i,\cdot} + w) + (1-\alpha) z \cdot r_{i,\cdot} \;, \quad \forall i \in [n] \enspace.
\end{equation*}
\ifthenelse{\boolean{arxiv}}{
}
{}
The map $T$ is $\alpha$-contracting and its unique fixed point $w$ is such that
$(w, (1-\alpha) z w)$ is solution of the ergodic dynamic programming equation~\eqref{eqn:progdyn}.
\end{prop}
\begin{proof}
 The set $\{\nu \;\text{st} \; \alpha \nu + (1-\alpha) z \in \mathcal{P}_i\}$ is a~set of probability measures so 
it is clear that $T$ is $\alpha$-contracting. Let $w$ be its fixed point. For all $i \in [n]$, 
\begin{equation*}
w_i =  \max_{\nu \;\text{st} \; \alpha \nu + (1-\alpha) z \in \mathcal{P}_i} 
\alpha \nu (r_{i,\cdot} + w) + (1-\alpha) z \cdot r_{i,\cdot}
= \max_{\nu \in \mathcal{P}_i} 
 \nu (r_{i,\cdot} + w) - (1-\alpha) z w
\end{equation*}
We get equation \eqref{eqn:progdyn} with constant $(1-\alpha) z w$.
\end{proof}
\begin{remark} \label{rem:discount}
 $T$ is the dynamic programming operator of a~total reward discounted problem with discount rate
 $\alpha$ and rewards $r'_{i,j} = r_{i,j} + \frac{1-\alpha}{\alpha} \sum_{l\in [n]} z_{l} r_{i,l}$ 
 for transition from $i$ to~$j$ (cf. Remark~\ref{rem:teleportPrice}).
\end{remark}

\begin{remark} \label{rem:fixPointBias}
 The fixed point found is just the mean reward before teleportation at the optimum (see Definition~\ref{defn:v}, Section~\ref{sec:shape}) .
\end{remark}

We can then solve the dynamic programming equation~(\ref{eqn:progdyn}) and so the
PageRank Optimization Problem~\eqref{eqn:ideal} or~\eqref{eqn:idealDisc} 
 with local constraints by value iteration.

The algorithm starts with an~initial potential function $w$, 
scans repeatedly the pages and updates $w_i$ when $i$ is the current
page according to $w_i \leftarrow T_i(w)$
until convergence is reached.
Then $(w , (1-\alpha) z w)$ is solution of the ergodic dynamic programming equation~(\ref{eqn:progdyn}) and the optimal linkage strategy is recovered by selecting the
maximizing $\nu$ at each page. 

Thanks to the damping factor $\alpha$, the iteration can be seen to be $\alpha$-contracting. 
Thus the algorithm converges in a~number of steps independent of the
dimension of the web graph.

For the evaluation of the dynamic programming operator, one can use a~linear program
using to the description of the actions by facets. It is however usually possible to develop algorithms
much faster 
than linear programming. We describe here a~greedy algorithm for the discrete PageRank Optimization problem.
The algorithm is straightforward
if the set of obligatory links $\mathcal{O}_i$ is empty (Propositions~\ref{prop:emptySet} and~\ref{prop:emptySet2}), 
so we only describe it in the other case. In 
Algorithm~\ref{alg:greedy}, $J$ represents the set of facultative hyperlinks activated. We
initialize it with the empty set and we augment it
with the best hyperlink until 
it is not valuable any more to add a hyperlink.

 \begin{algorithm}[htpb] 
\caption{Evaluation of the dynamic programming operator in the discrete problem} 
\begin{algorithmic}[1]  \label{alg:greedy} 
\STATE Initialization: $J \leftarrow \emptyset$ and $k \leftarrow 1$ 
\STATE Sort $(w_l + r_{i,l})_{l \in \mathcal{F}_i}$ in decreasing order and let $\psi : \{1,\ldots, \abs{\mathcal{F}_i} \} \rightarrow \mathcal{F}_i$ be the sort  function so that 
$w_{\psi(1)} + r_{i,\psi(1)} \geq \dots\geq
w_{\psi(|\mathcal{F}_i|)} + r_{i,\psi(|\mathcal{F}_i|)}$.
\WHILE{$\frac{1}{\abs{J}+\abs{\mathcal{O}_i}} \sum_{l \in J \cup \mathcal{O}_i} (w_l + r_{i,l}) < w_{\psi(k)} + r_{i,\psi(k)}$ 
and $k \leq \mathcal{F}_i$} 
\STATE $J \leftarrow J \cup \{ \psi(k) \}$ and $k \leftarrow k+1$ 
\ENDWHILE 
\STATE $T_i(w)= \alpha \frac{1}{\abs{J}+\abs{\mathcal{O}_i}} \sum_{l \in J \cup \mathcal{O}_i}(w_l + r_{i,l})+ (1-\alpha) \sum_{l \in [n]} z_l r_{i,l}$ 
 \end{algorithmic}  
\end{algorithm} 

\begin{prop} 
When the constraints of the Discrete PageRank Optimization problem~\eqref{eqn:idealDisc} are defined by obligatory, facultative and
forbidden links, the greedy algorithm (Algorithm~\ref{alg:greedy}) started at page $i$ returns $T_i(w)$ as defined in Proposition~\ref{prop:contrrate}. 
\end{prop} 
\begin{proof} 
 The local constraints are obviously respected by construction.  
At the end of the loop, we have the best choice of facultative outlinks from page~$i$ with exactly $\abs{J}$ outlinks.  
But as $\frac{1}{\abs{J}+ \abs{\mathcal{O}_i}} \sum_{l \in J \cup \mathcal{O}_i} (w_l + r_{i,l}) \geq w_j + r_{i,j} \Leftrightarrow  
\frac{1}{\abs{J}+\abs{\mathcal{O}_i}} \sum_{l \in J \cup \mathcal{O}_i} (w_l + r_{i,l}) \geq \frac{1}{\abs{J}+\abs{\mathcal{O}_i}+1} \sum_{l \in J \cup \mathcal{O}_i \cup \{j\}} (w_l + r_{i,l})$, 
the~sorting implies that we have the best choice of outlinks. 
\end{proof} 


\begin{remark}
 A straightforward modification of the greedy algorithm can handle 
a upper or a~lower limit on the number of links on a~given page. 
\end{remark} 

\begin{prop} \label{prop:greedy}
An $\epsilon$ approximation of the Discrete PageRank Optimization Problem~\eqref{eqn:idealDisc} 
with only local constraints can be done in time
\[
\text{\Large{O}} \Big ( \frac{\log(\epsilon)}{\log(\alpha)} \sum_{i \in [n]} \abs{\mathcal{O}_i} + \abs{\mathcal{F}_i} \log(\abs{\mathcal{F}_i}) \Big )
\] 
\end{prop}
\begin{proof}
The value of the PageRank optimization problem
is $(1-\alpha)z w$ where $w=T(w)$.
Thus it is bounded by $(1-\alpha) \norm{z}_1 \norm{w}_{\infty}=(1-\alpha) \norm{w}_{\infty}$.
The greedy algorithm described in the preceding paragraph evaluates  
the $i$th coordinate of the dynamic programming operator $T$  
\ifthenelse{\boolean{arxiv}}{in a time bounded by}{in time}
$O(\abs{\mathcal{O}_i} + \abs{\mathcal{F}_i} \log(\abs{\mathcal{F}_i}))$  
(by performing a~matrix-vector product and a~sort).  
Thus it evaluates the dynamic programming operator \ifthenelse{\boolean{arxiv}}{in a time bounded by}{in time}
 $O \left (\sum_{i \in [n]} \abs{\mathcal{O}_i} + \abs{\mathcal{F}_i} \log(\abs{\mathcal{F}_i}) \right )$. 

Now, if we normalize the rewards and if we begin the value iteration with
$w^0 = 0$, the initial error is less than $1$ in sup-norm.
The fixed point iteration reduces this error by at least $\alpha$, so we have to find $k \in \mathbb{N}$ such that $\alpha^k \leq \epsilon$.
With $k \geq \frac{\log(\epsilon)}{\log(\alpha)}$, the result holds.
\end{proof}

This result should be compared to PageRank computation's complexity by the power method~\cite{PRcomputing}, which is 
$O \left ( \frac{\log(\epsilon)}{\log(\alpha)} \sum_{i \in [n]} \abs{\mathcal{O}_i} + \abs{\mathcal{F}_i} \right )$. 


\section{General shape of an~optimized web site} \label{sec:shape}


We now use the previous model to identify the features of
optimal link strategies. In~particular, we shall identify 
circumstances under which there is always one ``master'' page,
to which all other pages should link.

As in the work of De Kerchove, Ninove and Van Dooren~\cite{NinKer-PRopt}, we 
shall use the mean reward before teleportation to study the optimal outlink strategies.
\begin{defn} \label{defn:v}
Given a~stochastic matrix $P$, the \NEW{mean reward before teleportation}
is given by
$ 
v(P) := (I_n - \alpha S)^{-1} \bar{r}
$, 
where $\bar{r}_i = \sum_j P_{i,j} r_{i,j}$.
\end{defn}
 Recall that $S$ is the original matrix (without damping factor),

\begin{prop} \label{prop:masterPage}
 Suppose the instantaneous reward $r_{i,j}$ only depends
 on the current page $i$ ($r_{i,j}=r'_i$). Denote $v(P)$ be the mean reward before teleportation 
(Definition \ref{defn:v}). 
Then $P$ is an optimal link strategy of the continuous PageRank Optimization problem~\eqref{eqn:ideal} if and only if
\[
 \forall i \in [n], \quad P_{i,\cdot} \in \arg\max_{\nu \in \mathcal{P}_i} \nu v(P)
\]
\end{prop}
\begin{proof}
We have $Pv(P)=v(P)-r'+\pi(P) r'$. Thus, using $\nu e=1$, the condition of the proposition is
equivalent to $\forall i \in [n], v_i(P)+\pi(P) r'=\max_{\nu \in \mathcal{P}_i} \nu (v(P)+r_i' e)$. By Proposition~\ref{prop:progdynprinc}, this
means that $v(P)$ is the bias of Equation~\eqref{eqn:progdyn} and that $P$ is an optimal outlink strategy.
\end{proof}

\begin{remark}
Proposition~\ref{prop:masterPage} shows that if $P$ is any optimal outlink strategy, 
at every page $i$, the transition probability $P_{i,\cdot}$ must maximize the same linear function.
\end{remark}
\begin{remark}
 If two pages have the same constraint sets, then they have the same optimal outlinks,
independently of their PageRank. This is no more the case with coupling constraints.
\end{remark}

For the discrete PageRank
Optimization problem, we have a~more precise result:
\begin{thm}[Master Page] \label{thm:masterPageDisc} 
Consider the Discrete PageRank Optimization problem~\eqref{eqn:idealDisc}
with constraints defined by given sets of obligatory, facultative and forbidden links. 
Suppose the instantaneous reward $r_{i,j}$ only depends
 on the current page $i$ ($r_{i,j}=r'_i$). Let $v$ be the mean reward before teleportation 
(Definition \ref{defn:v}) at the optimum. 
Then any optimal link strategy must choose 
for every controlled page $i$ all the facultative links $(i,j)$ 
such that $v_j > \frac{v_i - r_i}{\alpha}$ and any combination 
of facultative links such that $v_j = \frac{v_i - r'_i}{\alpha}$. 
Moreover, all optimal link strategies are obtained in this way. 

In~particular, every controlled page should point to the page with the highest 
mean reward before teleportation, as soon as it is allowed to. 
We call it the ``master page''. 
\end{thm} 

\begin{proof} 
By Remark~\ref{rem:fixPointBias}, we know that the mean reward before teleportation 
at the optimum is a~fixed point of the dynamic programming operator. 
In~particular, it is invariant by the application of the greedy algorithm (Algorithm~\ref{alg:greedy}). 
Moreover, by Proposition~\ref{prop:contrrate},  
the mean reward before teleportation at the optimum is unique. 

Thus, any optimal strategy must let the mean reward before teleportation invariant
by the greedy algorithm. 
When there is no obligatory link from page~$i$,
either a~link $(i,j)$ is selected and $v_i=\alpha v_j + r'_i$ 
or no link is selected and $v_i=\sum_{k \in [n]} z_k v_k +r'_i> \alpha v_j +r'_i$ for all facultative link~$(i,j)$.
When there is at least one obligatory link,
from Line~3 of the greedy algorithm, we know that, denoting $J$ the set of activated links, all the links $(i,j)$ 
verifying $\frac{1}{\abs{J}+\abs{\mathcal{O}_i}} \sum_{l \in J \cup \mathcal{O}_i} v_l + r'_{i} < v_j + r'_{i}$,
 must be activated.
This can be rewritten as $v_j > \frac{v_i - r'_i}{\alpha}$ because 
$v_i=\alpha \frac{1}{\abs{J}+\abs{\mathcal{O}_i}} \sum_{l \in J \cup \mathcal{O}_i} v_l + r'_i$.

Finally, activating any combination
of\ifthenelse{\boolean{arxiv}}{ the}{} facultative links such that $v_j = \frac{v_i - r'_i}{\alpha}$
gives the same mean reward before teleportation.
\end{proof}

The theorem is illustrated in~Example~2 (Section~\ref{sec:solvelocal}) and Figure~\ref{fig:smallExample}.

\begin{figure}
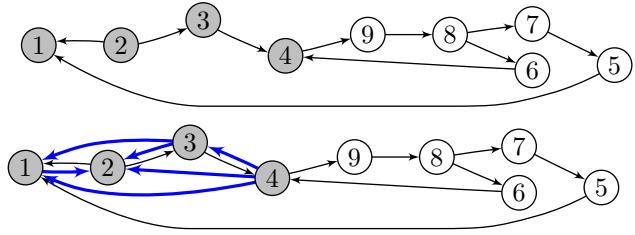
 
\centering
\ifthenelse{\boolean{notikz}}
{
\includegraphics[width=23em]{graphExampleDisc}
\includegraphics[width=23em]{graphExampleDiscOptimized}
}
{
\small{
\begin{tikzpicture}[>=latex', join=bevel, scale=0.25]
\pgfsetlinewidth{0.5bp}
\pgfsetcolor{black}
  \draw [->] (793bp,112bp) .. controls (812bp,102bp) and (840bp,89bp)  .. (870bp,74bp);
  \draw [->] (176bp,94bp) .. controls (195bp,98bp) and (221bp,106bp)  .. (251bp,118bp);
  \draw [->] (871bp,49bp) .. controls (852bp,39bp) and (823bp,26bp)  .. (796bp,19bp) .. controls (732bp,1bp) and (714bp,0bp)  .. (646bp,0bp) .. controls (274bp,0bp) and (274bp,0bp)  .. (274bp,0bp) .. controls (192bp,0bp) and (102bp,46bp)  .. (48bp,78bp);
  \draw [->] (124bp,97bp) .. controls (106bp,99bp) and (82bp,99bp)  .. (52bp,98bp);
  \draw [->] (672bp,111bp) .. controls (690bp,114bp) and (714bp,117bp)  .. (744bp,121bp);
  \draw [->] (548bp,108bp) .. controls (566bp,108bp) and (590bp,108bp)  .. (620bp,108bp);
  \draw [->] (744bp,56bp) .. controls (679bp,59bp) and (510bp,69bp)  .. (424bp,74bp);
  \draw [->] (670bp,98bp) .. controls (689bp,89bp) and (716bp,78bp)  .. (746bp,65bp);
  \draw [->] (296bp,115bp) .. controls (314bp,105bp) and (341bp,94bp)  .. (372bp,81bp);
  \draw [->] (423bp,83bp) .. controls (442bp,87bp) and (467bp,94bp)  .. (497bp,102bp);
\begin{scope}
  \pgfsetstrokecolor{black}
  \pgfsetfillcolor{lightgray}
  \filldraw (26bp,92bp) ellipse (26bp and 26bp);
  \draw (26bp,92bp) node {1};
\end{scope}
\begin{scope}
  \pgfsetstrokecolor{black}
  \pgfsetfillcolor{lightgray}
  \filldraw (274bp,130bp) ellipse (26bp and 26bp);
  \draw (274bp,130bp) node {3};
\end{scope}
\begin{scope}
  \pgfsetstrokecolor{black}
  \pgfsetfillcolor{lightgray}
  \filldraw (150bp,92bp) ellipse (26bp and 26bp);
  \draw (150bp,92bp) node {2};
\end{scope}
\begin{scope}
  \pgfsetstrokecolor{black}
  \draw (894bp,62bp) ellipse (26bp and 26bp);
  \draw (894bp,62bp) node {5};
\end{scope}
\begin{scope}
  \pgfsetstrokecolor{black}
  \pgfsetfillcolor{lightgray}
  \filldraw (398bp,76bp) ellipse (26bp and 26bp);
  \draw (398bp,76bp) node {4};
\end{scope}
\begin{scope}
  \pgfsetstrokecolor{black}
  \draw (770bp,124bp) ellipse (26bp and 26bp);
  \draw (770bp,124bp) node {7};
\end{scope}
\begin{scope}
  \pgfsetstrokecolor{black}
  \draw (770bp,54bp) ellipse (26bp and 26bp);
  \draw (770bp,54bp) node {6};
\end{scope}
\begin{scope}
  \pgfsetstrokecolor{black}
  \draw (522bp,108bp) ellipse (26bp and 26bp);
  \draw (522bp,108bp) node {9};
\end{scope}
\begin{scope}
  \pgfsetstrokecolor{black}
  \draw (646bp,108bp) ellipse (26bp and 26bp);
  \draw (646bp,108bp) node {8};
\end{scope} 
\end{tikzpicture}
\\[2mm]
\begin{tikzpicture}[>=latex', join=bevel, scale=0.25]
\pgfsetlinewidth{0.5bp}
\pgfsetcolor{black}
  \draw [->] (793bp,112bp) .. controls (812bp,102bp) and (840bp,89bp)  .. (870bp,74bp);
  \draw [->] (176bp,94bp) .. controls (195bp,98bp) and (221bp,106bp)  .. (251bp,118bp);
  \draw [->] (871bp,49bp) .. controls (852bp,39bp) and (823bp,26bp)  .. (796bp,19bp) .. controls (732bp,1bp) and (714bp,0bp)  .. (646bp,0bp) .. controls (274bp,0bp) and (274bp,0bp)  .. (274bp,0bp) .. controls (192bp,0bp) and (102bp,46bp)  .. (48bp,78bp);
  \draw [->] (124bp,97bp) .. controls (106bp,99bp) and (82bp,99bp)  .. (52bp,98bp);
  \draw [->] (672bp,111bp) .. controls (690bp,114bp) and (714bp,117bp)  .. (744bp,121bp);
  \draw [->] (548bp,108bp) .. controls (566bp,108bp) and (590bp,108bp)  .. (620bp,108bp);
  \pgfsetcolor{blue}
  \draw [->,very thick] (248bp,132bp) .. controls (218bp,134bp) and (167bp,135bp)  .. (124bp,127bp) .. controls (102bp,123bp) and (78bp,114bp)  .. (50bp,103bp);
  \pgfsetcolor{black}
  \draw [->] (744bp,56bp) .. controls (679bp,59bp) and (510bp,69bp)  .. (424bp,74bp);
  \pgfsetcolor{blue}
  \draw [->,very thick] (248bp,128bp) .. controls (229bp,123bp) and (203bp,115bp)  .. (173bp,104bp);
  \draw [->,very thick] (373bp,70bp) .. controls (324bp,60bp) and (215bp,41bp)  .. (124bp,57bp) .. controls (102bp,60bp) and (78bp,69bp)  .. (50bp,81bp);
  \draw [->,very thick] (52bp,86bp) .. controls (70bp,85bp) and (94bp,85bp)  .. (124bp,87bp);
  \pgfsetcolor{black}
  \draw [->] (670bp,98bp) .. controls (689bp,89bp) and (716bp,78bp)  .. (746bp,65bp);
  \pgfsetcolor{blue}
  \draw [->,very thick] (376bp,91bp) .. controls (358bp,100bp) and (331bp,112bp)  .. (300bp,124bp);
  \pgfsetcolor{black}
  \draw [->] (296bp,115bp) .. controls (314bp,105bp) and (341bp,94bp)  .. (372bp,81bp);
  \draw [->] (423bp,83bp) .. controls (442bp,87bp) and (467bp,94bp)  .. (497bp,102bp);
  \pgfsetcolor{blue}
  \draw [->,very thick] (372bp,78bp) .. controls (327bp,81bp) and (237bp,86bp)  .. (176bp,90bp);
\begin{scope}
  \pgfsetstrokecolor{black}
  \pgfsetfillcolor{lightgray}
  \filldraw (26bp,92bp) ellipse (26bp and 26bp);
  \draw (26bp,92bp) node {1};
\end{scope}
\begin{scope}
  \pgfsetstrokecolor{black}
  \pgfsetfillcolor{lightgray}
  \filldraw (274bp,130bp) ellipse (26bp and 26bp);
  \draw (274bp,130bp) node {3};
\end{scope}
\begin{scope}
  \pgfsetstrokecolor{black}
  \pgfsetfillcolor{lightgray}
  \filldraw (150bp,92bp) ellipse (26bp and 26bp);
  \draw (150bp,92bp) node {2};
\end{scope}
\begin{scope}
  \pgfsetstrokecolor{black}
  \draw (894bp,62bp) ellipse (26bp and 26bp);
  \draw (894bp,62bp) node {5};
\end{scope}
\begin{scope}
  \pgfsetstrokecolor{black}
  \pgfsetfillcolor{lightgray}
  \filldraw (398bp,76bp) ellipse (26bp and 26bp);
  \draw (398bp,76bp) node {4};
\end{scope}
\begin{scope}
  \pgfsetstrokecolor{black}
  \draw (770bp,124bp) ellipse (26bp and 26bp);
  \draw (770bp,124bp) node {7};
\end{scope}
\begin{scope}
  \pgfsetstrokecolor{black}
  \draw (770bp,54bp) ellipse (26bp and 26bp);
  \draw (770bp,54bp) node {6};
\end{scope}
\begin{scope}
  \pgfsetstrokecolor{black}
  \draw (522bp,108bp) ellipse (26bp and 26bp);
  \draw (522bp,108bp) node {9};
\end{scope}
\begin{scope}
  \pgfsetstrokecolor{black}
  \draw (646bp,108bp) ellipse (26bp and 26bp);
  \draw (646bp,108bp) node {8};
\end{scope} 
\end{tikzpicture}
}
}
\caption{Maximization of the sum of the PageRank values of the colored pages.
 Top: obligatory links; self links are forbidden; all other links are facultative. 
Bottom: bold arcs represent an optimal linking strategy. Page 4
points to all other controlled pages and Page~1, the master page, is pointed to by all other controlled pages. No
facultative link towards an external page is selected. 
}
\label{fig:smallExample}
\end{figure}

\begin{example}
The following simple counter examples show respectively that the conditions
that instantaneous rewards only depend on the current page and
that there are only local constraints are useful in the preceding
theorem.

Take a~two pages web graph without any design constraint.
Set $\alpha=0.85$, $z=(0.5, 0.5)$ and the reward per click 
$r= {\small \begin{bmatrix}1 & 10 \\ 2 & 2\end{bmatrix}}$. Then $v=(39.7, 35.8)$, Page~$2$ should
link to Page~$1$ but Page~$1$ should link to Page~$2$ because $39.7+1 \leq 35.8+10$.

Take the same graph as in preceding example. Set $r'=(0,1)$
and the coupling constraint that $\pi_1 \geq \pi_2$.
Then every optimal strategy leads to $\pi_1=\pi_2=0.5$.
This means that there is no ''master'' page because both
pages must be linked to in order to reach $\pi_i=0.5$.
\end{example}

\begin{remark} \label{rem:degenerate}
If every controlled page is allowed to point to every page, 
as in Figures~\ref{graph:local} and~\ref{graph:discrete}, there is a~master page to which every page should point.
Actually, knowing that the optimal solutions are degenerate might be of interest
to detect link spamming (or avoid being classified as a~link spammer).
The result of Proposition~\ref{prop:masterPage} and Theorem~\ref{thm:masterPageDisc} can be related to~\cite{linkspamalliances}, where the authors
show various optimal strategies for link farms: patterns with every page linking to
one single master page also appear in their study. 
We also remark that in~\cite{collusionTopologies}, 
the authors show that making collusions is a~good way to improve PageRank. We give
here the page with which one should make a~collusion. 
\end{remark}



\begin{remark} \label{rem:directMasterPage}
If there exists a~page with maximal reward 
in which all the hyperlinks can be changed, then this page is the master page.
It will have a~single hyperlink, pointing to the second highest page
in terms of mean reward before teleportation.
\end{remark}

%



\begin{remark} \label{rem:incompCrawling}
Major search engines have spent lots of efforts on crawling the web to
discover web pages and the hyperlinks between them. They can thus
compute accurately the PageRank. A search engine optimization team may not have
such a~database available. If one can program a~crawler to get a~portion of the web graph
or download some datasets of reasonable size for free
(\cite{webbase} for instance), these are still incomplete crawlings when compared to the search engine's.

We denote by $v$ and $\tilde{v}$ the mean reward before teleportation of
respectively the search engine's web graph and the trucated web graph. 
Let $I$ be the set of pages of interest, that is the pages containing
or being pointed to by a facultative link. 
 We denote by 
$R$ the length of a~shortest path from a~page in $I$ to an~uncrawled page.
We can easily show that if there are no page without outlink, then for all $i$ in $I$, 
$\abs{v_i-\tilde{v}_i} \leq \alpha^{R+1}\frac{2}{1-\alpha} \norm{\bar{r}}_{\infty}$.
\ifthenelse{\boolean{arxiv}}
{

When there are pages without outlink, the problem is more technical.
A possible approach to deal with it is to use the non-compensated PageRank~\cite{Mathieu-phd}.
}
{
}
\end{remark}

\section{PageRank Optimization with coupling constraints} \label{sec:coupling}

\subsection{Reduction of the problem with coupling constraints to constrained Markov decision processes}  


From now on, we have studied discrete or continuous PageRank Optimization problems 
but only with local constraints. 
We consider in this section the following 
PageRank Optimization problem~(\ref{eqn:ideal}) with ergodic (linear in the occupation measure) coupling constraints:
\ifthenelse{\boolean{arxiv}}
{
\begin{equation*}
  \max_{\pi,P} \sum_{i,j} \pi_i P_{i,j} r_{i,j}  \text{ st:}
\end{equation*}
\begin{equation} \label{eqn:coupling}
\pi P = \pi\;,\; \pi \in \Sigma_n\;, \;P_{i,\cdot} \in \mathcal{P}_i, \forall i \in [n]\\
\end{equation}
\begin{equation*}
 \sum_{i,j} \pi_i P_{i,j} d^k_{i,j} \leq V^k, \forall k \in K
\end{equation*}
}
{
\begin{equation} \label{eqn:coupling}
\max_{\pi,P} \big \{ U(\pi,P) \; \; ; \; \pi = \pi  P ,\; \pi \in \Sigma_n, 
\; P_{i,\cdot} \in \mathcal{P}_i, \forall i \in [n] , \; \sum_{i,j} \pi_i P_{i,j} d^k_{i,j} \leq V^k, \forall k \in K \big \}
\end{equation}
}

Examples of ergodic coupling constraints are given in Section~\ref{sec:cont}. 

When coupling constraints are present, the previous standard ergodic control
model is no longer valid, but we can use instead the theory of {\em constrained} Markov decision processes. 
We refer the reader to~\cite{Alt-CMDP} for more background.
In addition to the instantaneous reward $r$, which is used to define
the ergodic functional which is maximized, we now consider 
a finite family of cost functions $(d^k)_{k \in K}$, 
together with real constants $(V^k)_{k \in K}$, which will be
used to define the ergodic constraints.
The ergodic constrained Markov decision problem consists in finding an
admissible control strategy $(\nu_t)_{t \geq 0}$, $\nu_t \in A_{X_t} , \forall t \in \geq 0$, maximizing:
\begin{equation} \label{eqn:CMDP}
 \liminf_{T \rightarrow +\infty} \frac{1}{T}
 \mathbb{E}(\sum_{t=0}^{T-1} r(X_t,\nu_{t}))
\end{equation}
under the $|K|$ \NEW{ergodic constraints}
\begin{equation*}
 \limsup_{T \rightarrow +\infty} \frac{1}{T}
 \mathbb{E}(\sum_{t=0}^{T-1} d^k(X_t,\nu_{t})) \leq V^k, \; \forall k \in K
\end{equation*}
where the controlled process $(X_t)_{t \geq 0}$ is such that
\ifthenelse{\boolean{arxiv}}
{
\begin{equation*}
  \mathbb{P}( X_{t+1}=j | X_t , \nu_t ) = p(j|X_t,\nu_t) \enspace .
\end{equation*}
}
{
$ 
  \mathbb{P}( X_{t+1}=j | X_t , \nu_t ) = p(j|X_t,\nu_t)
$. 
}

Theorem~4.1 in~\cite{Alt-CMDP} shows that one can restrict to stationary
Markovian strategies and Theorem~4.3 in the same book gives an~equivalent
formulation of the ergodic constrained Markov decision problem~\eqref{eqn:CMDP} as a linear program.
When $A_i=\extr(\mathcal{P}_i)$, $r(i,a)=\sum_{j\in [n]} r_{i,j}a_j$, $d^k(i,a)=\sum_{j\in [n]} d^k_{i,j}a_j$ and $p(j|i,a)=a_j$ (see Proposition~\ref{prop-local}),
it is easy to see that this linear program is equivalent to:
\begin{equation} \label{eqn:LP}
 \max_{\rho} \; \big \{ \; \sum_{i,j \in [n]} \rho_{i,j} r_{i,j}  \text{ st: } 
\rho \in \mathcal{R} \; \text{ and } 
 \sum_{i,j \in [n]} \rho_{i,j} d^k_{i,j} \leq V^k, \; \forall k \in K \; \big \}
\end{equation}
where $\mathcal{R}$ is the image of $\prod_{i \in [n]} \mathcal{P}_i$ by the correspondence of Proposition~\ref{prop:homeo}.
The set $\mathcal{R}$ is a~polyhedron, as soon as every $\mathcal{P}_i$ is a~polyhedron (Proposition~\ref{prop:convex+}).

%

Following the correspondence
discussed in Proposition~\ref{prop:homeo}, we can see that the linear Problem~(\ref{eqn:LP})
is just the reformulation of Problem~(\ref{eqn:coupling}) in terms of occupation measures when
we consider total income utility~\eqref{eqn:income}. 



The last result of this section gives a~generalization to nonlinear utility functions:
\begin{prop} \label{prop:loc->glob}
 Assume that the utility function\ifthenelse{\boolean{arxiv}}{ $U$}{} can be written as $U(P) = W(\rho)$
where $W$ is concave, that the local constraints are convex in $P$
and that the coupling constraints are ergodic. Then, the PageRank
Optimization problem~\eqref{eqn:coupling} is equivalent to a~concave
programming problem in the occupation measure $\rho$, 
from which $\epsilon$-solutions can be found in 
polynomial time.
\end{prop}
\begin{proof}
{From} Proposition~\ref{prop:convex+}, we know that the set
of locally admissible occupation measures is convex. Adding ergodic (linear in the occupation measure)
constraints preserves this convexity property. So the whole optimization problem is concave.
Finally, Theorem~5.3.1 in~\cite{Nem-modernConvOpt} states that $\epsilon$-solutions can be found in 
polynomial time.
\end{proof}

In~particular, the (global) optimality of a~given occupation measure
can be checked by the first order optimality conditions which are
standard in convex analysis. 

\begin{remark} \label{rem:entropy}
Proposition~\ref{prop:loc->glob} applies
in~particular if $W$ is a~relative entropy utility function, ie
$ 
W(\rho) = - \sum_{i,j\in[n]} \rho_{ij}\log (\rho_{ij}/\mu_{ij}) 
$, 
where parameters $\mu_{ij}>0$ (the reference measure)
are given.

If we choose to minimize the entropy function on the 
whole web graph, we recover the TrafficRank algorithm~\cite{Tomlin-HOTS}. 
When we control only some of the hyperlinks whereas the weights of the 
others are fixed, the solution of the optimization problem gives the 
webmaster the weights that she should set to her hyperlinks in order 
to have an~entropic distribution of websurfers on her website, interpreted 
as a~fair distribution of websurfers.  
\end{remark}


\ifthenelse{\boolean{optimalitycondition}}
{
In the next section, we extend the first order optimality conditions
to the formulation in probability transitions, in order to 
get a characterization of the optimal linking strategies in the constrained PageRank Optimization problem.


\subsection{Optimality condition} \label{sec:condopt}

The following shows that the mean reward before teleportation (Definition~\ref{defn:v}) 
determines
the derivative of the utility function. Recall that the {\em tangent cone}
$\mathrm{T}_{X}(x)$ of the set $X$ at point $x$ is
the closure of the set of vectors $q$ such that $x+t q \in X$ for $t$ small enough.
\begin{prop} \label{prop:derivee}
The derivative of total utility function~\eqref{eqn:income}
is such that for all $Q \in \mathrm{T}_{\mathcal{P}}(P)$,
\begin{equation*}
\left \langle \mathrm{D}U(P) , Q \right \rangle = \sum_{i,j} (v_j(P) + r_{i,j}) \pi_i (P) Q_{i,j} 
\end{equation*}
where $v(P)$ is the mean reward before teleportation,
$\pi(P)$ is the invariant measure of $P$ and $\langle \cdot,\cdot\rangle$ is the standard (Frobenius) scalar product on $n\times n$ matrices. 
\end{prop}

\begin{proof}
 We have $U (P) =\sum_{i,j} \pi_i (P) P_{i,j} r_{i,j} = \pi \bar{r}$ and 
$\pi = \pi P = \pi (\alpha S + (1-\alpha) e z)$.
As~$\pi e = 1$, we have an~explicit expression for $\pi$ as function of $P$:
$ 
 \pi(P) = (1-\alpha) z (I_n - P + (1-\alpha) e z)^{-1}
$. 
The result follows from derivation of $\pi(P)\bar{r}$. We need to derive a product, to derive an inverse ($\left \langle \mathrm{D}(A\mapsto A^{-1}) , H \right \rangle = - A^{-1} H A^{-1}$) and the expression of\ifthenelse{\boolean{arxiv}}{ the mean reward before teleportation}{} $v(P)=(I_n-P+(1-\alpha)ez)^{-1}\bar{r}$.
\end{proof}

The next theorem, which involves the mean reward before
teleportation, shows that although
the continuous constrained pagerank optimization problem is non-convex, the
first-order necessary optimality condition is also sufficient.

\begin{thm}[Optimality Condition] \label{thm:condopt}
Suppose that the sets $\mathcal{P}_i$ defining local constraints are all closed convex sets,
that the coupling constraints are given by the ergodic costs functions $d^k$, $k \in K$ and that
the utility function is total income utility. Denote $\mathcal{P}^d$ be the admissible set 
and $v(P)$ the mean reward before teleportation (Definition~\ref{defn:v}).
We introduce the set of saturated constraints $K_{sat} = \{ k \in K | \sum_{i,j} d^k_{i,j} \pi_i P_{i,j} = V^k \}$ and
\ifthenelse{\boolean{arxiv}}{we introduce the numbers}{we denote} $D^k_{i,j}=\pi_i d^k_{i,j} + \pi d^k (I-\alpha S)^{-1} e_i P_{i,j}$.
Then the tangent cone of $\mathcal{P}^d$ at $P$ is 
$ 
 \mathrm{T}_{\mathcal{P}^d}(P)= \Big \{ Q \in \prod_{i \in [n]} \mathrm{T}_{\mathcal{P}_i}(P_{i,\cdot})\; 
| \; \forall k \in K_{sat} \; ,\; \langle D^k, Q \rangle \leq 0  \Big \}
$ 
and $P^* \in \mathcal{P}^d$ is the optimum of the continuous PageRank Optimization problem~\eqref{eqn:ideal} 
with ergodic coupling constraints if and only if:
\begin{equation*}
\forall Q \in \mathrm{T}_{\mathcal{P}^d}(P^*)\; , \quad
\sum_{i,j \in [n]} \pi_i (v_j(P^*) + r_{i,j}) Q_{i,j} \leq 0
\end{equation*}
\end{thm}
\begin{proof}
Let us consider the birational change of variables of
Proposition~\ref{prop:homeo}. As all the occupation measures considered are irreducible,
its Jacobian is invertible\ifthenelse{\boolean{arxiv}}{ at any admissible point.}{.} 
Thus, we can use the results of Section~6.C in~\cite{Rockafellar-varAnalysis}.
Denote $\mathcal{P}= \prod_{i \in [n]} \mathcal{P}_i$, with tangent cone $\mathrm{T}_{\mathcal{P}}(P)= \prod_{i \in [n]} \mathrm{T}_{\mathcal{P}_i}(P_{i,\cdot})$,
and $\mathcal{R}=f^{-1}(\mathcal{P})$. We have
$ 
\mathrm{T}_{\mathcal{R}^d}(\rho)=\Big \{ \sigma \in \mathrm{T}_{\mathcal{R}}(\rho)\; 
| \; \forall k \in K_{sat} \;, \; \langle d^k, \sigma \rangle \leq 0 \Big \}
$ 
and
$
\mathrm{T}_{\mathcal{P}^d}(P)=\Big \{ Q \in \mathbb{R}^{n \times n}\; 
| \; \nabla f^{-1} Q \in \mathrm{T}_{\mathcal{R}^d}(f^{-1}(P)) \Big \} 
$. 

$\nabla f^{-1} Q \in \mathrm{T}_{\mathcal{R}^d}(f^{-1}(P))$ first means that 
$\nabla f^{-1} Q \in \mathrm{T}_{\mathcal{R}}(f^{-1}(P))$ which can also be written as
$\nabla f \nabla f^{-1} Q = Q \in \mathrm{T}_{\mathcal{P}}(P)$.
The second condition is $ \forall k \in K_{sat} , \langle d^k,\nabla f^{-1} Q \rangle \leq 0$.
As $(f^{-1}(P))_{i,j}=\rho_{i,j}=\pi_i P_{i,j}$, we have
$ 
 (\nabla f^{-1} Q)_{i,j} = \sum_{k,l} Q_{k,l} (P_{k,l} \frac{\partial \pi_k}{P_{i,j}}+\pi_k \delta_{ik} \delta_{jl})
$. 
Thanks to the expression the derivative of the utility function and of $\frac{\partial \pi_k}{P_{i,j}}=\pi_i e_j(I-\alpha S)^{-1} e_k$ 
both given in Proposition~\ref{prop:derivee}, we get the expression stated in the theorem.

 By Proposition~\ref{prop:loc->glob}, the PageRank optimization problem is a~concave
programming problem in~$\rho$ and so, the first
order (Euler) optimality condition guarantees the global optimality
of a~given measure. Thus, every stationary point for the continuous PageRank Optimization
problem is a~global maximum when written in transition probabilities also.
\end{proof}
}
{}

\subsection{A Lagrangian relaxation scheme to handle coupling constraints between pages} \label{sec:lag}

The PageRank Optimization Problem with 
''ergodic'' coupling constraints~(\ref{eqn:coupling}) 
may be solved by off the shelve simplex or interior points solvers.
However, such general purpose solvers may be too slow, or too memory
consuming, to solve the largest web instances.

The following proposition yields an~algorithm that decouples the computation
effort due to complexity of the graph and due to coupling constraints.

\begin{prop} \label{prop:Lag}
The PageRank Optimization problem with $K$ ''ergodic'' coupling constraints~(\ref{eqn:coupling})
can be solved by a~Lagrangian relaxation scheme, in which the dual function
\ifthenelse{\boolean{arxiv}}
{
and one of its subgradient
\begin{equation*}
        \theta(\lambda) = \max_{\rho \in \mathcal{R}} \left< r, \rho \right> - \sum_{k \in K} \lambda_k (\left< d^k, \rho \right> - V^k)
\end{equation*}
\begin{equation*}
  \frac{\partial \theta}{\partial \lambda_k}(\lambda) = \left< d^k, \rho^*(\lambda) \right> - V^k
\end{equation*}
}
{defined by 
$ 
        \theta(\lambda) = \max_{\rho \in \mathcal{R}} \left< r, \rho \right> - \sum_{k \in K} \lambda_k (\left< d^k, \rho \right> - V^k)
$ 
and one 
 subgradient
$ 
  \frac{\partial \theta}{\partial \lambda_k}(\lambda) = \left< d^k, \rho^*(\lambda) \right> - V^k
$ 
}
are evaluated by dynamic programming and $\rho^*(\lambda)$ is a~maximizer of the expression defining $\theta(\lambda)$.
\end{prop}

\begin{proof}
This is a~simple application of Lagrange multipliers theory, 
see~\cite{lemarechal-2001} Theorem 21 and Remark 33 for instance. Here we relax the coupling constraints in the problem 
written with occupation measures~(\ref{eqn:LP}).
We solve the dual problem, namely we minimize the dual function $\theta$ on $\mathbb{R}^K_{+}$. The 
value of this dual problem is the same as the value of the constrained primal problem and we can get a~solution of the primal
problem since there is no duality gap.
\end{proof} 
 
We have implemented a~bundle high level algorithm, in which the dual function 
is evaluated at each step by running a~value iteration algorithm, 
for a~problem with modified reward. By comparison with the unconstrained 
case, the execution time is essentially multiplied by the number of iterations 
of the bundle algorithm. 

\section{Experimental results} \label{sec:exp}

\subsection{Continuous problem with local constraints only} 
We have tried our algorithms on a~crawl on eight New Zealand Universities
available at~\cite{NZ2006}.
There are 413,639 nodes and 2,668,244 links in the graph. The controlled set we have chosen
is the set of pages containing ''maori'' in their url. There are 1292 
of them.
We launched the experiments in a~sequential manner on a~personal computer with Intel Xeon CPU at 2.98~Ghz 
and wrote the code in Scilab language.

Assume that the webmasters controlling these pages cooperate and agree to change at most 20\% of the links' weight to improve the PageRank, being understood
that self-links are forbidden
(skeleton constraint, see Section~\ref{sec:cont}).
The algorithm launched on the optimization of the sum of the PageRanks of the controlled pages (calculated
with respect to the crawled graph only, 
not by the world wide graph considered by Google)
ran 27~seconds.

The optimal strategy returned is that every controlled page except itself should link with 20\% weight to
 \quotepage{maori-oteha.massey.ac.nz/te\_waka.htm}. That page should link to 
\ifthenelse{\boolean{arxiv}}{\quotepage{maori-oteha.massey.ac.nz/tewaka/about.htm}}{\quotepage{maori-oteha. massey.ac.nz/tewaka/about.htm}}.
The sum of PageRank values goes from 0.0057 to 0.0085. 

Hence, by uniting, this team of webmasters would improve the sum of their PageRank scores of 49\%. 
Remark that all the pages point to the same page 
(except itself because self-links are forbidden).
The two best pages to point to are in fact part of a~''dead end'' of the web graph
containing only pages with maximal reward. A random surfer can only escape from this
area of the graph by teleporting, which makes the mean reward before teleportation maximal. 

\subsection{Discrete problem} \label{sec:expDisc}

On the same data set, we have considered the discrete optimization problem.
The set of obligatory links is the initial set of links. 
We have then selected 2,319,174 
facultative links on the set of controlled pages of preceding section. 

Execution time took 81 
seconds with the polyhedral approach of Section~\ref{subsec-value} (60 iterations).
We compared our algorithm with an~adaptation of the graph augmentation approach of~\cite{Blondel-PRopt}  
to total utility: this algorithm took 460 seconds (350 iterations) 
for the same precision. 
The optimal strategy is to
add no link that goes out of the website but get the internal
link structure a~lot denser. From 12,288 internal links, the optimal strategy is to add 962873 
internal links. Finally, 98.2\% 
of the links are internal links and there is
a mean number of links per page of 770. 
 The sum of PageRank values jumps
from 0.0057 to 0.0148.

Here, as the weights of the links cannot be changed, the webmaster can hardly
force websurfers to go to dead ends. But she can add so many links that
websurfers get lost in the labyrinth of her site and do not find the outlinks,
even if they were obligatory.

\subsection{Coupling linear constraints} \label{sec:expCoupl}

We would like to solve the discrete optimization problem
of the preceding section with two additional coupling constraints. We require
that each visitor coming on one of the pages of the team has a~probability to leave
 the set of pages of the team on next step of 40\% (coupling conditional probability constraint, see Section~\ref{sec:cont}).
We also require that the sum of PageRank values of the home pages of the 10 universities considered
remains at least equal to their initial value after the optimization (effective frequency constraint). 

In the case of constrained Markov decision processes, optimal strategies
are usually randomized strategies. This means that the theory cannot directly deal
with discrete action sets. 
Instead, we consider the continuous problem with
the polytopes of uniform transition measures as local admissible sets,
i.e.\ we relax the discrete pattern. 
Thus by the Lagrangian scheme of Proposition~\ref{prop:Lag}, 
we get an~upper bound on the optimal objective and we have a~lower bound 
for any admissible discrete transition matrix. 

The initial value is 0.0057 and the Lagrangian relaxation scheme gives an~upper bound of 0.00769. 
Computation took 675~s (11 high level iterations).
During the course of the Lagrangian relaxation scheme, all intermediate solutions
are discrete and three of them satisfied the coupling constraints. 
The best of them corresponds to a~sum of PageRank values of 0.00756.
Thus we have here a~duality gap of at most 1.7\%. 
In general, 
the intermediate discrete solutions need not satisfy the coupling constraints
and getting an admissible discrete solution may be difficult.

The discrete transition matrix found suggests to add 124,328 internal links
but also 11,235 external links. As in Section~\ref{sec:expDisc},
lots of links are added, but here there are also external links.

The bounding technique proposed here
can also be adapted to PageRank optimization problem
with mutual exclusion constraints.
It may also be possible to
use it to design a~branch and bound algorithm to solve the problem
exactly thanks to the bounds found.

\section*{Conclusion}

We have presented in this paper a~general framework to study
the optimization of PageRank. Our results apply to a~continuous 
problem where the webmaster can choose the weights of the hyperlinks on her pages and to
the discrete problem in which a~binary decision must be taken to decide whether a~link is present.
We have shown that the Discrete PageRank Optimization problem without coupling constraints can
be solved by reduction to a concisely described~relaxed continuous problem. 
We also showed that the continuous Pagerank optimization problem is 
polynomial time solvable, even with coupling constraints.
\ifthenelse{\boolean{arxiv}}
{

}
{
}
We gave scalable algorithms which 
rely on an 
ergodic control model and on dynamic programming techniques.
The first one, which applies to problems with local design constraints, is a fixed point scheme whose 
convergence rate shows that optimizing PageRank is not much more
complicated than computing it.
The second algorithm, which handles coupling constraints, is still efficient when the number of coupling constraints remains small.
\ifthenelse{\boolean{arxiv}}
{

}
{
}
We have seen that the mean reward before teleportation gives a~total order
of preference in pointing to a~page or an~other. This means
that pages high in this order concentrate many inlinks from controlled pages.
This is a~rather degenerate strategy when we keep in mind that a~web site should convey information.
Nevertheless, the model allows to address more 
complex problems, for instance with coupling constraints,  
in order to get less trivial solutions.  
\ifthenelse{\boolean{arxiv}}
{

}
{
}
This work may be useful to understand link
spamming, to price internet advertisements or, by changing the objective
function, to design web sites with other goals like fairness or usefulness.
The latter is the object of further research.

\bibliographystyle{IEEEtran.bst}
\bibliography{IEEEabrv,pagerank}

\end{document}